\documentclass[10pt,a4paper]{amsart}%
\usepackage{amsfonts}
\usepackage{amssymb}
\usepackage[utf8]{inputenc}
\usepackage{amsmath}
\usepackage{graphicx}%
\providecommand{\U}[1]{\protect\rule{.1in}{.1in}}
\newtheorem{theorem}{Theorem}[section]
\newtheorem{proposition}[theorem]{Proposition}
\newtheorem{lemma}[theorem]{Lemma}
\newtheorem{corollary}[theorem]{Corollary}

\theoremstyle{remark}

\newcommand{\remove}[1]{ }

\def\R{\mathbb R}
\def\be{\begin{equation}}
\def\ee{\end{equation}}
\def\ba{\begin{eqnarray}}
\def\ea{\end{eqnarray}}

\def\vf{\varphi}
\def\al{\alpha}
\def\ii{\int\!\!\!\int}

\numberwithin{equation}{section}
\begin{document}
\title{Internal controllability of the Korteweg-de Vries equation on a bounded domain}
\author[Capistrano--Filho]{R. A. Capistrano--Filho}
\address{Instituto de Matem\'{a}tica, Universidade Federal do Rio de Janeiro, C.P.
68530 - Cidade Universit\'{a}ria - Ilha do Fund\~{a}o, 21941-909 Rio de
Janeiro (RJ), Brazil and Institut Elie Cartan, UMR 7502 UHP/CNRS/INRIA, BP
70239, 54506 Vand\oe uvre-les-Nancy Cedex, France}
\email{capistrano@im.ufrj.br}
\author[Pazoto]{A. F. Pazoto}
\address{Instituto de Matem\'{a}tica, Universidade Federal do Rio de Janeiro, C.P.
68530 - Cidade Universit\'{a}ria - Ilha do Fund\~{a}o, 21941-909 Rio de
Janeiro (RJ), Brazil}
\email{ademir@im.ufrj.br}
\author[Rosier]{L. Rosier}
\address{Institut Elie Cartan, UMR 7502 UHP/CNRS/INRIA, BP 70239, 54506
Vand\oe uvre-les-Nancy Cedex, France}
\email{rosier@iecn.u-nancy.fr}

\begin{abstract}
This paper is concerned with the control properties of the Korteweg-de Vries (KdV)  equation posed on a bounded interval 
with a distributed control. When the control region is an arbitrary open subdomain, we prove the null controllability of the KdV equation
by means of a new Carleman inequality. As a consequence, we obtain a regional controllability result, the state function being controlled on the
left part of the complement of the control region. Finally,  when the control region is a neighborhood of the right endpoint, 
an exact controllability result in a weighted $L^2$-space is also established.
\end{abstract}
\maketitle

\section{Introduction\label{Sec0}}

The Korteweg--de Vries (KdV) equation can be written
\[
u_t+u_{xxx}+u_x+uu_x=0,
\]
where $u=u(t,x)$ is a real-valued function of two real variables $t$ and $x$, and $u_t=\partial u/\partial t$, etc. 
The equation was first derived by Boussinesq \cite{Boussinesq} and Korteweg-de Vries  \cite{Korteweg} as
a model for the propagation of water waves along a channel. 
The equation furnishes also a very useful approximation model 
in nonlinear studies whenever one wishes to include and balance a 
weak nonlinearity and weak dispersive effects.  In particular, the equation is now
commonly accepted as a mathematical model for the unidirectional propagation of small amplitude long waves
in nonlinear dispersive systems.

The KdV equation has been intensively studied from various aspects of mathematics, including the 
well-posedness, the existence and stability of solitary waves, the integrability, the long-time behavior, etc. (see e.g.
\cite{Kenig,miura}).
The practical use of the KdV equation does not always involve the pure initial value problem. 
In numerical studies, one is often interested in using a finite interval (instead of the whole line) with three
boundary  conditions. 

Here, we shall be concerned with the control properties of KdV, the control acting
through a forcing term $f$ incorporated in the equation:

\be
u_{t}+u_{x}+u_{xxx}+ uu_{x}=f, \quad t\in [0,T], \ x\in[0,L], \qquad + \ \text{ b.c.} 
\label{a1}
\ee

Our main purpose is to see whether one can force the solutions of \eqref{a1}
to have certain desired properties by choosing an appropriate control input
$f$. The focus here is on the {\em controllability} issue: 

\noindent\textit{Given an initial state $u_{0}$ and a
terminal state $u_{1} $ in a certain space, can one find an
appropriate control input $f$ so that the equation (\ref{a1}) admits a
solution $u$ which equals $u_{0}$ at time $t=0$ and  $u_{1}$ at time
$t=T$?}

If one can always find a control input $f$ to guide the system described by
(\ref{a1}) from any given initial state $u_{0}$ to any given terminal state
$u_{1}$, then the system (\ref{a1}) is said to be {\em exactly
controllable}. If the system can be driven, by means of a control $f$, from
any state to the origin (i.e. $u_{1}\equiv0$), then one says that system \eqref{a1} is
{\em null controllable}.

The study of the controllability and stabilization of the KdV equation started with the works
of Russell and Zhang  \cite{Russell1}  for a system with periodic
boundary conditions and an internal control. Since then, both the controllability and the stabilization have
been intensively studied. (We refer the reader to \cite{RZsurvey} for a survey of the results up to 2009.)  
In particular, the exact boundary controllability of KdV on
a finite domain was investigated in e.g. \cite{cerpa,cerpa1,coron,GG,GG1,Rosier,Rosier2,Zhang2}.
Most of those works were concerned with the following system
\begin{equation}
\left\{
\begin{array}
[c]{lll}%
u_{t}+u_{x}+u_{xxx}+uu_{x}=0 &  & \text{in }(  0,T)  \times(
0,L)  \text{,}\\
u(  t,0)  =g_{1}(t),\,u(  t,L)  =g_{2}(t),\,u_{x}(
t,L)  =g_{3}(t) &  & \text{in }(  0,T)  
\end{array}
\right.  \label{int1-bc}%
\end{equation}
in which the boundary data $g_1,g_2,g_3$ can be chosen as control inputs. 
System \eqref{int1-bc} was first studied by Rosier \cite{Rosier} considering
only the control input $g_{3}$  (i.e. $g_1=g_2=0$). It was shown in \cite{Rosier} 
that the exact controllability of the linearized system holds in $L^2(0,L)$ if, and only if,
$L$ does not belong to the following countable set of
{\em critical lengths}
\begin{equation}
\mathcal{N}:=\left\{  \frac{2\pi}{\sqrt{3}}\sqrt{k^{2}+kl+l^{2}}%
\,:k,\,l\,\in\mathbb{N}^{\ast}\right\}  . \label{critical}%
\end{equation}
The analysis developed in \cite{Rosier} shows that when the linearized system is
controllable, the same is true for the nonlinear one. Note that the converse is false, as it was
proved in \cite{cerpa,cerpa1,coron} that the (nonlinear) KdV equation is controllable even when $L$ is a critical length. 
The existence of a discrete set of critical lengths for which the exact controllability of the linearized equation fails was also
noticed by Glass and Guerrero in \cite{GG1} when $g_2$ is taken as control input (i.e. $g_1=g_3=0$).  
 Finally, it is worth mentioning the result by Rosier \cite{Rosier2} and Glass and Guerrero \cite{GG} for which
 $g_1$ is taken as control input (i.e. $g_2=g_3=0$). They
proved that system \eqref{int1-bc} is then null controllable, but not exactly controllable, because of the strong smoothing effect.

By contrast, the mathematical theory pertaining to the study of the internal
controllability in a bounded domain is considerably less advanced. As far as we know, the null
controllability problem for system \eqref{a1} was only addressed in \cite{GG}
when the control acts in a neighborhood of the left endpoint. 
On the other hand, the exact controllability results in \cite{Laurent,Russell1} were obtained on a periodic domain.

The aim of this paper is to address the controllability issue for the KdV equation on a 
bounded domain with a distributed control. Our first main result is a null controllability result valid for any localization of the control region. 
Actually, a controllability to the trajectories is established:

\begin{theorem}
\label{nullsysnon} Let $\omega =(l_1,l_2)$ with $0<l_1<l_2<L$, and let $T>0$. 
For $\bar{u}_{0}\in L^{2}(  0,L)  $, let
$\bar{u}\in C^0( [0,T];L^{2}(  0,L)  )  \cap
L^{2}(  0,T;H^1(  0,L)  )  $ denote the solution of%
\begin{equation}
\left\{
\begin{array}
[c]{lll}%
\bar{u}_{t}+\bar{u}_{x}+\bar{u}\, \bar{u}_{x}+\bar{u}_{xxx}=0 &  & \text{in
}(  0,T)  \times(  0,L)  \text{,}\\
\bar{u}(  t,0)  =\bar{u}(  t,L)  =\bar{u}_{x}(
t,L)  =0 &  & \text{in }(  0,T)  \text{,}\\
\bar{u}(  0,x)  =\bar{u}_{0}(  x)  &  & \text{in
}(  0,L)  \text{.}%
\end{array}
\right.  \label{ncn1}%
\end{equation}
Then there exists $\delta>0$ such that for any $u_{0}\in L^{2}(0,L)  $ satisfying $\left\Vert u_{0}-\bar{u}_{0}\right\Vert
_{L^{2}(  0,L)  }\leq\delta$, there exists $f\in 
L^{2}(  (0,T)  \times\omega)  $ such that the solution $u\in C^{0}(  [ 0,T]  ; L^{2}(  0,L)  )\cap L^2 (0,T,H^1(0,L))   $ 
of%
\begin{equation}
\left\{
\begin{array}
[c]{lll}%
u_{t}+u_{x}+uu_{x}+u_{xxx}=1_{\omega}f(  t,x)  &  & \text{in
}(  0,T)  \times(  0,L)  \text{,}\\
u(  t,0)  =u(  t,L)  =u_{x}(  t,L)  =0 &  &
\text{in }(  0,T)  \text{,}\\
u(  0,x)  =u_{0}(  x)  &  & \text{in }(
0,L)  \text{,}%
\end{array}
\right.  \label{ncn2}%
\end{equation}
satisfies $u(  T,\cdot)  =\bar{u}(  T,\cdot)  $ in
$(  0,L)  $.
\end{theorem}

The null controllability is first established for a linearized system
\begin{equation}
\left\{
\begin{array}
[c]{lll}%
u_{t}+\left(  \xi u\right)  _{x}+u_{xxx}= 1_{\omega} f &  & \text{in }\left(  0,T\right)
\times\left(  0,L\right)  \text{,}\\
u\left(  t,0\right)  =u\left(  t,L\right)  =u_{x}\left(  t,L\right)  =0 &  &
\text{in }\left(  0,T\right)  \text{,}\\
u\left(  0,x\right)  =u_{0}\left(  x\right)  &  & \text{in }\left(
0,L\right)  \text{,}%
\end{array}
\right.  \label{int}%
\end{equation}
by following the classical duality approach (see \cite{dolecki,Lions1}), which
reduces the null controllability of \eqref{int} to  an observability inequality for the solutions of 
the adjoint system. To prove the observability inequality, we derive a new Carleman estimate with an internal
observation in $(0,T)\times (l_1,l_2)$ and use some interpolation arguments inspired by those in \cite{GG}, where the authors
derived a similar result when the control acts on a neighborhood on the left endpoint (that is, $l_1 =0$). 
The null controllability is extended to the nonlinear system by applying Kakutani fixed-point theorem. 

The second problem we address is related to the exact internal
controllability of system \eqref{a1}. As far as we know, the same problem
was studied only in \cite{Laurent,Russell1} in a periodic
domain $\mathbb{T}$ with a distributed control of the form
\[
f(x,t)=(Gh)(x,t):= g(x)(  h(x,t)-\int_{\mathbb{T}}g(y) h(y,t) dy),
\]
where $g\in C^\infty (\mathbb T)$ was such that $\{g>0\} = \omega$ and $\int_{\mathbb T} g(x)dx=1$, and 
the function $h$ was considered as a new control input. Here, we shall consider the system 
\begin{equation}
\left\{
\begin{array}
[c]{lll}%
u_{t}+u_{x}+uu_{x}+u_{xxx}= f  &  & \text{in }(  0,T)  \times(  0,L)  \text{,}\\
u(  t,0)  =u(  t,L)  =u_{x}(  t,L)  =0 &  & \text{in }(  0,T)  \text{,}\\
u(  0,x)  =u_{0}(  x)  &  & \text{in }( 0,L)  \text{.}%
\end{array}
\right.  \label{S1}%
\end{equation}

As the smoothing effect is different from those in a periodic domain, the results in this paper turn out to be very different from those in 
\cite{Laurent,Russell1}. 
First, for a controllability result in $L^2(0,L)$, the control $f$ has to be taken in the space $L^2(0,T,H^{-1}(0,L))$. Actually, with   
any control $f\in L^2(0,T,L^2(0,L))$, the solution of \eqref{S1} starting from $u_0=0$ at $t=0$ would remain in $H^1_0(0,L)$ (see \cite{GG}). 
On the other hand, as for the boundary control, the localization of the distributed control plays a role in the results. 
 
When the control acts in a neighborhood of $x=L$,  we obtain the exact controllability in the weighted Sobolev
space $L^{2}_{\frac{1}{L-x}  dx}$ defined as
\[
L^{2}_{\frac{1}{L-x}  dx}:=\{  u\in L^1_{loc}(0,L);
\int_{0}^{L}\frac{\left\vert u(x)
\right\vert ^{2}}{ L-x  } dx<\infty\} .
\]
 
More precisely, we shall obtain the following result:
\begin{theorem}
\label{thmB}
Let $T>0$, $\omega=(l_1,l_2)=(L-\nu ,L)$ where $0<\nu <L$.
Then, there exists $\delta>0$ such that for any $u_{0}$, $u_{1}\in L^2_{  \frac{1}{L-x}  dx}$
with
\[
\left\Vert u_{0}\right\Vert _{L^2_{ \frac{1}{L-x}  dx} } \leq\delta \ \text{ and } \  
\left\Vert u_{1}\right\Vert _{L^2_{  \frac{1}{L-x}  dx}}\leq\delta ,
\]
one can find a control input $f\in L^{2}(  0,T;H^{-1}(0,L))  $ with $\text{supp} (f)\subset (0,T) \times \omega$ 
such that the solution $u\in C^0([0,L],L^2 (0,L) )
\cap L^2(0,T,H^1(0,L))$ of (\ref{S1}) satisfies $u(T,.)  =u_{1}\,\text{ in } (0,L)$ and $u\in C^0([0,T],L^2_{  \frac{1}{L-x}  dx} )$. Furthermore, 
$f\in L^2_{ (T-t) dt}(0,T,L^2(0,L))$.  
\end{theorem}

Actually, we shall have to investigate the well-posedness of 
the linearization of \eqref{S1} in the space $L^2_{ \frac{1}{L-x}dx}$ and  the well-posedness
 of the (backward) adjoint system in  the ``dual space'' $L^2_{ (L-x) dx}$.
To do this, we shall follow some ideas borrowed from \cite{goubet}, where the well-posedness was investigated in 
the weighted space $L^2_{\frac{x}{L-x}dx}$. 
The needed observability inequality is obtained by the standard compactness-uniqueness argument and some unique continuation property. 
The exact controllability is extended to the nonlinear system by using the contraction mapping principle. 

When the control is acting far from the endpoint $x = L$, i.e. in some interval $\omega =(l_1,l_2)$ with $0<l_1<l_2<L$, then there is no chance to control 
exactly the state function on $(l_2,L)$ (see e.g. \cite{Rosier2}). However, it is possible to control the state function on $(0,l_1)$, so that a ``regional controllability''
can be established:

\begin{theorem}
\label{thmC}
Let $T>0$ and $\omega = (l_1,l_2)$ with $0<l_1<l_2<L$.  Pick any number $l_1'\in (l_1,l_2)$. Then there exists a number $\delta >0$ such that  for any
$u_0,u_1\in L^2(0,L)$ satisfying 
\[
||u_0||_{L^2(0,L)} \le \delta, \qquad ||u_1||_{L^2(0,L)} \le \delta ,
\]
one can find a control $f\in L^2(0,T,H^{-1}(0,L))$ with $\text{supp}(f)\subset (0,T)\times \omega$ such that the solution 
$u\in C^0([0,T],L^2(0,L)) \cap L^2(0,T,H^1 (0,L))$ of \eqref{S1} satisfies 
\be
\label{corS1}
u(T,x) =\left\{ 
\begin{array}{ll}
u_1(x) & \text{ if }  \ x\in (0,l_1');\\
0 &\text{ if } \ x\in (l_2,L).
\end{array}
\right. 
\ee
\end{theorem}
The proof of Theorem \ref{thmC} combines Theorem \ref{nullsysnon}, a boundary controllability result from \cite{Rosier}, and the use of a cutt-off function. 
Note that the issue whether $u$ may also be controlled in the interval $(l_1',l_2)$ is open. 

The paper is outlined as follows. In Section 2, we review some linear estimates  from \cite{GG,Rosier} that will be used thereafter. 
Section 3 is devoted to the proof of Theorems \ref{nullsysnon} and  \ref{thmC}. It contains the proof of a new Carleman estimate for the KdV equation
with some internal observation (Proposition \ref{prop10}). In Section 4 we prove the well-posedness of KdV in the weighted spaces $L^2_{xdx}$ and $L^2_{\frac{1}{L-x}dx}$ by  using semigroup theory, and derive Theorem \ref{thmB}. 
\section{Linear estimates\label{Sec1}}

We review a series of estimates for the system
\begin{equation}
\left\{
\begin{array}
[c]{lll}%
u_{t}+(  \xi u)  _{x}+u_{xxx}=  f(  t,x)  &  &
\text{in }(  0,T)  \times(  0,L)  \text{,}\\
u(  t,0)  =u(  t,L)  =u_{x}(  t,L)  =0 &  &
\text{in }(  0,T)  \text{,}\\
u(  0,x)  =u_{0}(  x)  &  & \text{in }(
0,L)  \text{}%
\end{array}
\right.  \label{1}%
\end{equation}
and its adjoint system.
Here $f= f (  t,x)  $ is a function which stands for 
the control of the system, and $\xi =\xi (t,x)$ is a given function. 

\subsection{The linearized KdV equation}

It was noticed in  \cite{Rosier} that the operator $A=-\dfrac{\partial^{3}%
}{\partial x^{3}}-\dfrac{\partial}{\partial x}$ with domain
\[
{\mathcal D}(  A)  =\left\{  w\in H^{3}(  0,L)  ; \ w(  0)
=w(  L)  =w_{x}(  L)  =0\right\}  \subseteq L^{2}(
0,L)
\]
is the infinitesimal generator of a strongly continuous semigroup of
contractions in $L^{2}(  0,L)  $. More precisely, the following result was established in \cite{Rosier}.

\begin{proposition}
\label{linkdv}Let $u_{0}\in L^{2}(  0,L)$, $\xi \equiv 1$  and $f\equiv0$. There
exists a unique (mild) solution $u$ of (\ref{1})
with%
\begin{equation}
u\in C([0,T];L^{2}(0,L))\cap L^2(0,T,H^1_0(0,L)).
\label{2}%
\end{equation}
Moreover, there exist positive constants $c_1$ and $c_2$ such that for all $u_0\in L^2(0,L)$ 
\begin{eqnarray}
\left\Vert u\right\Vert _{L^{2}(0,T;H^{1}(0,L))} + \left\Vert u_{x}(  .,0)  \right\Vert _{L^{2}(  0,T) }
&\leq& c_1\left\Vert u_{0}\right\Vert _{L^{2}(  0,L)  }, \label{4}\\
\left\Vert u_{0}\right\Vert _{L^{2}(  0,L)  } ^2
&\leq& \frac{1}%
{T}\left\Vert u\right\Vert _{L^{2}(  0,T;L^{2}(  0,L)
)  }^{2}+c_{2}\left\Vert u_{x}( .,0)  \right\Vert
_{L^{2}(  0,T)  }^{2}\text{.} \label{6}%
\end{eqnarray}
If in addition $u_{0}\in D(  A)  $, then (\ref{1}) has a unique
(classical) solution $u$ in the class%
\begin{equation}
u\in C([0,T];D(A))\cap C^{1}([0,T];L^{2}(0,L))\text{.} \label{3}%
\end{equation}
\end{proposition}

\subsection{The modified KdV equation}
We introduce a system related to the adjoint system to \eqref{1}, namely
\begin{equation}
\left\{
\begin{array}
[c]{lll}%
-v_{t}-  \xi v  _{x}-v_{xxx}=f &  & \text{in }(  0,T)
\times(  0,L)  \text{,}\\
v(  t,0)  =v(  t,L)  =v_{x}(  t,0)  =0 &  &
\text{in }(  0,T)  \text{,}\\
v(  T,x)  =0 &  & \text{in }(  0,L) ,%
\end{array}
\right.  \label{mod}%
\end{equation}
for which we review some estimates borrowed from \cite{GG}. 
\subsubsection{Energy Estimates}
We introduce the following spaces%
\begin{equation}%
\begin{array}
[c]{lll}%
X_{0}:=L^{2}(  0,T;H^{-2}(  0,L)  )  \text{,} &  &
X_{1}:=L^{2}(  0,T;H_{0}^{2}(  0,L)  )  \text{,}\\
\tilde{X}_{0}:=L^{1}(  0,T;H^{-1}(  0,L)  )  \text{,} &
& \tilde{X}_{1}:=L^{1}(  0,T;(  H^{3}\cap H_{0}^{2})  (
0,L)  )  \text{,}%
\end{array}
\label{13}%
\end{equation}
and%
\begin{equation}%
\begin{array}
[c]{l}%
Y_{0}:=L^{2}(  (  0,T)  \times(  0,L)  )
\cap C^{0}(  \left[  0,T\right]  ;H^{-1}(  0,L)  )
\text{,}\\
Y_{1}:=L^{2}(  0,T;H^{4}(  0,L)  )  \cap C^{0}(
\left[  0,T\right]  ;H^{3}(  0,L)  )  \text{.}%
\end{array}
\label{14}%
\end{equation}
The spaces $X_{0},X_{1},\tilde{X}_0,\tilde{X}_1,Y_0$, and $Y_1$ are equipped with their natural norms. For instance, the spaces
$Y_{0}$ and $Y_{1}$ are equipped with the norms%
\[
\left\Vert w\right\Vert _{Y_{0}}:=\left\Vert w\right\Vert _{L^{2}(
(  0,T)  \times(  0,L)  )  }+\left\Vert
w\right\Vert _{L^{\infty}(  0,T;H^{-1}(  0,L)  )  }%
\]
and%
\[
\left\Vert w\right\Vert _{Y_{1}}:=\left\Vert w\right\Vert _{L^{2}(
0,T;H^{4}(  0,L)  )  }+\left\Vert w\right\Vert _{L^{\infty
}(  0,T;H^{3}(  0,L)  )  }\text{.}%
\]
For $\theta\in\left[  0,1\right]  $, we define the complex interpolation
spaces (see \cite{Berg} and \cite{Lions})%
\[
X_{\theta}=(  X_{0},X_{1})  _{\left[  \theta\right]  }\text{,
}\tilde{X}_{\theta}=(  \tilde{X}_{0},\tilde{X}_{1})  _{\left[
\theta\right]  }\text{ and }Y_{\theta}=(  Y_{0},Y_{1})  _{\left[
\theta\right]  }\text{.}%
\]
Then,%
\begin{equation}%
X_{1/4}=L^{2}(0,T; H^{-1}(0,L) ) , \quad
\tilde{X}_{1/4}=L^{1}(  0,T;L^2(  0,L)  )
\label{14_3}%
\end{equation}
and
\begin{equation}
Y_{1/4}=L^{2}(  0,T;H^{1}(  0,L)  )  \cap C^{0}(
\left[  0,T\right]  ;L^{2}(  0,L)  )  \text{.} \label{14_2}%
\end{equation}
Furthermore,%
\begin{equation}%
X_{1/2}=L^{2}(  (  0,T)  \times(  0,L)  ) , \quad
\tilde{X}_{1/2}=L^{1}(  0,T;H_{0}^{1}(  0,L)  )
\label{a14_3}%
\end{equation}
and%
\begin{equation}
Y_{1/2}=L^{2}(  0,T;H^{2}(  0,L)  )  \cap C^{0}(
\left[  0,T\right]  ;H^{1}(  0,L)  )  \text{.} \label{14_4}%
\end{equation}

\begin{proposition} (\cite[Section 2.2.2]{GG})
\label{est1} Let $\xi \in Y_{\frac{1}{4}}$ and  $f\in X_{\frac{1}{4}} \cup \tilde X_{\frac{1}{4}} = L^{2}(  0,T;H^{-1}(  0,L)  )  \cup
L^{1}(  0,T;L^{2}(  0,L)  )  $. Then  the solution $v$ of
(\ref{mod}) belongs to $Y_{\frac{1}{4}}$,
and there exists some constant $C=C(||\xi ||_{Y_{ \frac{1}{4} } }  )>0$ such that %
\begin{equation}
\left\Vert v\right\Vert _{L^{\infty}(  0,T,L^{2}(  0,L)
)  }+\left\Vert v\right\Vert _{L^{2}(  0,T;H^{1}(  0,L)
)  }+\left\Vert v_{x}(  \cdot,L)  \right\Vert _{L^{2}(
0,T)  }\leq C(  \left\Vert \xi\right\Vert _{Y_{1/4}})
\left\Vert f\right\Vert _{L^{2}(  0,T;H^{-1}(  0,L)  )
} \label{11}%
\end{equation}
and%
\begin{equation}
\left\Vert v\right\Vert _{L^{\infty}(  0,T,L^{2}(  0,L)
)  }+\left\Vert v\right\Vert _{L^{2}(  0,T;H^{1}(  0,L)
)  }+\left\Vert v_{x}(  \cdot,L)  \right\Vert _{L^{2}(
0,T)  }\leq C(  \left\Vert \xi\right\Vert _{Y_{1/4}})
\left\Vert f\right\Vert _{L^{1}(  0,T;L^{2}(  0,L)  )
}\text{.} \label{12}%
\end{equation}
\end{proposition}
More can be said when $\xi\equiv 0$. Consider the following system%
\begin{equation}
\left\{
\begin{array}
[c]{lll}%
-v_{t}-v_{xxx}=g &  & \text{in }(  0,T)  \times(  0,L)
\text{,}\\
v(  t,0)  =v(  t,L)  =v_{x}(  t,0)  =0 &  &
\text{in }(  0,T)  \text{,}\\
v(  T,x)  =0 &  & \text{in }(  0,L)  .
\end{array}
\right.  \label{15}%
\end{equation}

\begin{proposition}  
(\cite[Section 2.3.1]{GG}.
\label{est2}
If $g\in X_{1}\cup\tilde{X}_{1}$, then $v$ $\in Y_{1}$ and there exists a
constant $C>0$ such that%
\begin{equation}
\left\Vert v\right\Vert _{Y_{1}}+\left\Vert v_{x}(  \cdot,L)
\right\Vert _{H^{1}(  0,T)  }\leq C
\left\Vert g\right\Vert _{X_{1}} \label{16}%
\end{equation}
and%
\begin{equation}
\left\Vert v\right\Vert _{Y_{1}}+\left\Vert v_{x}(  \cdot,L)
\right\Vert _{H^{1}(  0,T)  }\leq C\left\Vert g\right\Vert
_{\tilde{X}_{1}}\text{.} \label{17}%
\end{equation}
\end{proposition}

\begin{proposition}
(\cite[Section 2.3.2]{GG}.
\label{est3} If $g\in X_{1/2} \cup \tilde{X}_{1/2}$, then $v\in Y_{1/2}$, and there exists some constant $C>0$ such that %
\begin{equation}
\left\Vert v\right\Vert _{Y_{1/2}}+\left\Vert v_{x}(  \cdot,L)
\right\Vert _{H^{1/3}(  0,T)  }
 + \left\Vert v_{xx}(  \cdot,0)  \right\Vert _{L^{2}(  0,T)  }+\left\Vert v_{xx}(  \cdot,L)
\right\Vert _{L^{2}(  0,T)  }
\leq C\left\Vert g\right\Vert _{X_{1/2}} \label{18}%
\end{equation}
and%
\begin{equation}
\left\Vert v\right\Vert _{Y_{1/2}}+\left\Vert v_{x}(  \cdot,L)
\right\Vert _{H^{1/3}(  0,T)  }
 + \left\Vert v_{xx}(  \cdot,0)  \right\Vert
_{L^{2}(  0,T)  }+\left\Vert v_{xx}(  \cdot,L)
\right\Vert _{L^{2}(  0,T)  }
\leq C\left\Vert g\right\Vert
_{\tilde{X}_{1/2}}\text{.} \label{19}%
\end{equation}
\end{proposition}

\section{Null controllability results\label{Sec2}}
This section is devoted to the proof of Theorems \ref{nullsysnon}  and \ref{thmC}.  
\subsection{Null controllability of a linearized equation}

We first consider the system
\begin{equation}
\left\{
\begin{array}
[c]{lll}%
u_{t}+(  \xi u)  _{x}+u_{xxx}=1_{\omega} f (  t,x)  &  &
\text{in }(  0,T)  \times(  0,L)  \text{,}\\
u(t,0)  =u( t,L)  =u_{x}(t,L)  =0 &  &
\text{in }(  0,T)  \text{,}\\
u(  0,x)  =u_{0}(  x)  &  & \text{in }(
0,L),%
\end{array}
\right.  \label{nc1}%
\end{equation}
where $\xi =\xi (t,x)$ is a given function in $Y_{\frac{1}{4} } $, 
and $\omega=(  l_{1},l_{2}) \subset(  0,L)  $. Our aim is to prove the null controllability 
of \eqref{nc1}. To this end, we shall establish an observability inequality
for the corresponding adjoint system%
\begin{equation}
\left\{
\begin{array}
[c]{lll}%
-v_{t}-\xi(  t,x)  v_{x}-v_{xxx}=0 &  & \text{in }(
0,T)  \times(  0,L)  \text{,}\\
v(  t,0)  =v(  t,L)  =v_{x}(  t,0)  =0 &  &
\text{in }(  0,T)  \text{,}\\
v(  T,x)  =v_{T}(  x)  &  & \text{in }(
0,L) %
\end{array}
\right.  \label{nc2}%
\end{equation}
by using some Carleman inequality.

\subsubsection{Carleman inequality with internal observation}
Assume that $\omega =(l_1,l_2)$ with 
\[
0<l_1<l_2<L.
\]
Pick any function $\psi \in C^3([0,L])$ with 
\ba
&&\psi >0\textrm{ in } [0,L]; \label{C1}\\
&&|\psi '|>0, \ \psi ''<0,\ \textrm{ and } \psi '\psi ''' <0 \textrm{ in } [0,L]\setminus \omega ; \label{C2}\\
&&\psi '(0)<0 \textrm{ and } \psi '(L) >0; \label{C3}\\
&&\min_ {x\in [l_1,l_2]} \psi (x) =\psi (l_3)<\max_{x\in [l_1,l_2]}\psi (x) =\psi (l_1)=\psi (l_2),\quad \max_{x\in [0,L]}\psi (x)= \psi (0)=\psi (L)  \label{C4}\\
&&\psi (0)<\frac{4}{3} \psi (l_3),\label{C5}
\ea
for some $l_3\in (l_1,l_2)$. 
A convenient function $\psi$ is defined on $[0,L] \setminus \omega$ as
\[
\psi(x) =\left\{
\begin{array}{ll}
\varepsilon x^3-x^2-x+c_1&\textrm{ if } x\in [0,l_1],\\
-\varepsilon x^3 +ax +c_2&\textrm{ if } x\in [l_2,L]
\end{array} 
\right.
\]
with $\varepsilon,a,c_1,c_2>0$ conveniently chosen. Note first that $\psi (l_1)=\psi (l_2)$ and $\psi (0)=\psi (L)$ if, and only if, 
\[
a=(L-l_2)^{-1} (l_1^2+l_1-\varepsilon l_2^3-\varepsilon l_1^3 + \varepsilon L^3), \qquad c_1=c_2 -\varepsilon L^3 + aL.
\] 
Then $a>0$, $c_1-c_2>0$ and \eqref{C2}-\eqref{C3} hold provided that $0<\varepsilon \ll1$.  \eqref{C1} and \eqref{C5} hold for 
$c_2\gg1$. \eqref{C4} is easy to satisfy.

Set
\be
\label{CC2}
\vf (t,x) =\frac{\psi (x) }{t(T-t)} \cdot
\ee
For $f\in L^2(0,T; L^2(0,L))$ and $q_0\in L^2(0,L)$, let $q$ denote the solution of the system
\ba
q_t + q_{xxx} =f,&& t\in (0,T),\  x\in (0,L) , \label{A1}\\
q(t,0)=q(t,L)=q_x(t,L)=0,&&  t\in (0,T), \label{A2} \\
q(0,x) =q_0(x), &&  x\in (0,L). \label{A3}
\ea
Then the following Carleman inequality holds.
\begin{proposition}
\label{prop10} Pick any $T>0$. 
There exist two constants $C>0$ and $s_0 >0$ such that any $f\in L^2(0,T;L^2(0,L))$, any $q_0\in L^2(0,L)$ and any 
$s\ge s_0$,
the solution $q$ of
\eqref{A1}-\eqref{A3} fulfills
\begin{multline}
\int_0^T\!\! \int_0^L [s\vf |q_{xx}|^2 +(s\vf )^3 |q_x|^2 +(s\vf )^5 |q|^2 ]e^{-2s\vf} dx dt
+\int_0^T  [ (s\vf  |q_{xx}|^2 +(s\vf)^3 |q_{x}|^2)e^{-2s\vf }] _{\vert x=0} +  [s\vf  |q_{xx}|^2 e^{-2s\vf } ] _{\vert x=L} dt  \\
\le
C\left( \int_0^T\!\!\int_0^L |f|^2 e^{-2s\vf} dxdt + \int_0^T\!\! \int_{\omega}[s\vf |q_{xx}|^2 +(s\vf )^3 |q_x|^2 +(s\vf )^5 |q|^2 ]e^{-2s\vf} dx dt \right)
\label{C7}
\end{multline}
\end{proposition}
Actually, we shall need a Carleman estimate for \eqref{nc2} with the potential $\xi \in Y_{\frac{1}{4}}$. Let
\[
\tilde \varphi (t,x) = \varphi (t,L-x). 
\]
\begin{corollary}
\label{cor11}
Let $\xi \in Y_{\frac{1}{4}}$. Then there exist some positive constants $\tilde s_0= \tilde s_0(T, ||\xi ||_{Y_\frac{1}{4}} )$ and
$C=C(T,||\xi ||_{Y_{ \frac{1}{4} } }  )$ such that for all $s\ge \tilde s_0$ and  all $v_T\in L^2(0,L)$, the solution $v$ of  \eqref{nc2} fulfills
\begin{multline}
\int_0^T\!\! \int_0^L [s\tilde \vf | v_{xx}|^2 +(s\tilde \vf )^3 |v_{x}|^2 +(s\tilde \vf )^5 |v|^2 ]e^{-2s\tilde \vf} dxdt  \\
\le C  \int_0^T\!\! \int _{\omega}  [s\tilde \vf | v_{xx}|^2 +(s\tilde \vf )^3 |v_{x}|^2 +(s\tilde \vf )^5 |v|^2 ]e^{-2s\tilde \vf} dxdt  .
\label{D60}
\end{multline}
\end{corollary}
\noindent 
{\em Proof of Proposition \ref{prop10}.} 
We first assume that $q_0\in D(A)$ and that $f\in C([0,T];D(A))$, so that $q\in C([0,T];D(A)) \cap C^1([0,T];L^2(0,L))$. This will be
sufficient to legitimate the following computations.  The general case ($q_0\in L^2(0,L)$ and $f\in L^2(0,T;L^2(0,L))$) follows by density.
Indeed, if we set 
\[
p(t,x):= \sqrt{\varphi (t,l_3)} e^{-s\varphi (t,l_3)} q(t,x)
\]
then $p$ solves \eqref{A1}-\eqref{A3} with $q_0$ replaced by $0$, and $f$ replaced by
\[
\tilde f = \sqrt{\varphi (t,l_3)} e^{-s\varphi (t,l_3)} f + 
\left( \frac{1}{2} \varphi _t (t,l_3) \varphi ^{-\frac{1}{2}} (t,l_3) -s\varphi _t (t,l_3) \sqrt{\varphi (t,l_3)} \right) e^{-s\varphi (t,l_3)}   q, 
\]  
so that (with different constants $C$)
\[
\int_0^T\!\!\!\int_0^L \varphi |q_{xx}|^2 e^{-2s\varphi} dxdt \le C||p||^2_{L^2(0,T,H^2(0,L))} \le C || \tilde f ||^2_{L^2(0,T,L^2(0,L))}
\le C \big( ||f||^2_{L^2(0,T,L^2(0,L))} + ||q_0||^2_{L^2(0,L)} \big).
\]
Since 
\[
||q||^2_{L^2(0,T,H^1(0,L))} \le C \big( || f ||^2_{L^2(0,T,L^2(0,L))} + ||q_0||^2_{L^2(0,L)} \big)
\] 
we conclude that we can pass to the limit in each term in \eqref{C7}, if we take a sequence $\{ (q_0^n,f^n)\} _{n\ge 0} $ 
in ${\mathcal D}(A)\times C([0,T],{\mathcal D}(A))$ such that
$q_0^n\to q_0$ in $L^2(0,L)$ and $f^n\to f$ in $L^2(0,T,L^2(0,L))$.
 
Assume from now on that $q_0\in {\mathcal D}(A)$ and that $f\in C([0,T];{\mathcal D}(A))$.
Let $q$ denote the solution of \eqref{A1}-\eqref{A3}, and let $u=e^{-s\vf }q$, $w=e^{-s\vf } L(e^{s\vf } u)$, where
\be
L= \partial _t + \partial _x ^3. \label{C8}
\ee
Straightforward computations show that
\begin{multline}
w=Mu:= u_t + u_{xxx} + 3s\varphi _x u_{xx} 
+ (3s^2 \varphi _x ^2  + 3s \varphi _{xx}) u_x
+ (s^3\varphi _x^3 +3s^2 \varphi _x\varphi _{xx} + s(\varphi _t + \varphi _{xxx}))u.
\label{B1}
\end{multline}
Let $M_1$ and $M_2$ denote the (formal) self-adjoint and skew-adjoint parts of the operator $M$.
We readily obtain that
\ba
M_1u &:=& 3s (\varphi _x u_{xx} + \varphi _{xx} u_x) +
[s(\varphi _t + \varphi _{xxx} ) + s^3 \varphi _x^3 ]u, \label{B2}\\
M_2u &:=& u_t + u_{xxx} + 3s^2( \varphi _x ^2 u_x + \varphi _x \varphi _{xx}u).\label{B3}
\ea
On the other hand
\be
||w||^2 = ||M_1u||^2 + ||M_2 u||^2 + 2(M_1u,M_2u)
\label{C12}
\ee
where $(u,v)=\int_0^T\!\!\!\int_0^L uv dxdt$ and $||w||^2=(w,w)$.  From now on, for the sake of simplicity, we write
$\ii u$ (resp. $\int u\big\vert _0^L$) instead of $\int_0^T\!\!\! \int_0^L u(t,x)dxdt$ (resp.
$\int_0^T u(t,x) \big\vert _{x=0}^L dt$). The proof of the Carleman inequality follows the same pattern as in \cite{MOR,RZ09}. The first step provides
an exact computation of the scalar product $(M_1u,M_2u)$, whereas the second step gives the estimates obtained thanks to the (pseudoconvexity) conditions
\eqref{C1}-\eqref{C5}. \\

\noindent
{\sc Step 1. Exact computation of the scalar product in \eqref{C12}.}\\
Write
\[
2(M_1u,M_2u) =
2\ii [s(\vf _t + \vf _{xxx}) + s^3 \vf _x^3 ]u M_2u + 2\ii 3s(\vf _x u_{xx} + \vf _{xx} u_x) M_2 u
=:I_1+I_2.
\]
Let
\be
\al := s(\varphi _t + \varphi _{xxx}) + s^3 \varphi _x^3.
\label{C14}
\ee
Using \eqref{B3}, we decompose $I_1$ into
\[
I_1 = \ii 2\al uu_t + \ii 2\al u u_{xxx} 
+ 3s^2 \ii 2\al u (\vf _x^2 u_x + \vf _x\vf _{xx} u) .
\]
Integrating by parts with respect to $t$ or $x$, noticing that $u_{\vert x=0}=u_{\vert x=L} ={u_x}_{\vert x=L} =0$, and that
$u_{\vert t=0} = u_{\vert t=T}=0$ by \eqref{C1}, we obtain that
\ba
I_1 &=& -\ii \al _t u^2 +( 3\ii \al _x u_x^2 -\ii \al _{xxx} u^2 -\int \al u_x^2\big\vert _0^L)
-3 s^2 \ii \vf _x^2 \al _x u^2 \nonumber\\
&=& -\ii (\al _t + \al _{xxx} +3s^2 \vf _x^2 \al _x)u^2 + 3 \ii \al _x u_x^2 
-\int\al u_x^2 \big\vert _0^L. \label{B11}
\ea
Next, we compute
\[
I_2 = 2\ii 3s (\vf _x u_{xx} + \vf _{xx} u_x)(u_t + u_{xxx} +3s^2(\vf _x^2 u_x + \vf _x\vf_{xx}u)).
\]
Performing integrations by parts, we obtain successively
\begin{eqnarray*}
&& 2\ii (\vf _x u_{xx} + \vf _{xx} u_x) u_t = \ii \vf _{xt} u_x ^2, \\
&&2\ii (\vf _xu_{xx} + \vf _{xx} u_x)u_{xxx} =-3 \ii \vf _{xx} u_{xx} ^2 + \ii \vf _{4x} u_x^2
+\int(\vf _x u_{xx}^2 -\vf _{3x}u_x^2  + 2\vf _{xx}u_{xx}u_x)\big\vert _0^L,
\end{eqnarray*}
and
\[
2\ii (\vf _xu_{xx} + \vf _{xx} u_x)(\vf _x^2 u_x + \vf _x\vf _{xx} u) =-3\ii \vf _x^2 \vf _{xx} u_x^2
+\ii [(\vf _x^2\vf_{xx})_{xx} -(\vf_x \vf_{xx}^2)_x] u^2  + \int \vf _x^3 u_x^2 \big\vert _0^L.
\]
Thus
\begin{multline}
I_2 = -9s \ii  \vf _{xx} u_{xx}^2+
\ii [-27 s^3 \vf _x ^2 \vf _{xx} + 3s (\vf _{xt} + \vf _{4x} )] u_x^2 \\
+ \ii 9s^3 [ (\vf _x^2\vf _{xx})_{xx} -(\vf _x\vf_{xx}^2)_x ]   u^2 +
\int [ 3s (\vf _x u_{xx}^2 -\vf _{3x}u_x^2  + 2\vf _{xx}u_x u_{xx} ) 
+9s^3 \vf _x ^3 u_x^2 ] \big\vert _0^L
\label{B12}
\end{multline}
Gathering together \eqref{B11}-\eqref{B12}, we infer that
\ba
2(M_1u,M_2u) &=&
\ii [-(\al _t +\al _{xxx} +3s^2 \vf _x^2 \al _x ) + 9s^3 ((\vf _x^2\vf _{xx})_{xx} -(\vf _x \vf _{xx}^2)_x) ]u^2 \nonumber \\
&&\quad  +\ii [3\al _x -27 s^3 \vf _x^2 \vf _{xx} +3s (\vf _{xt} +\vf _{4x}) ] u_x^2  -9s \ii \vf _{xx} u_{xx}^2  \nonumber \\
&&\quad +\int [ 3s \vf _x u _{xx}^2 + (9s^3 \vf _x^3 -3s \vf _{xxx} -\al ) u_x^2  + 2 \vf _{xx} u_x u_{xx} ]  \big\vert _0^L  
 \label{B13}
\ea

\noindent
{\sc Step 2. Estimation of each term in \eqref{B13}.}\\
The estimates are given in a series of claims. \\
{\sc Claim 1.}
There exist some constants $s_1>0$ and $C_1>1$ such that for all $s\ge s_1$, we have
\begin{multline*}
\ii [-(\al _t +\al _{xxx} +3s^2 \vf _x^2 \al _x ) + 9s^3 ((\vf _x^2\vf _{xx})_{xx} -(\vf _x \vf _{xx}^2)_x) ]u^2
\ge C_1^{-1} \ii (s\varphi )^5 u^2  - C_1 \int_0^T\!\!\!\int_\omega (s\varphi )^5 u^2.
\end{multline*}
From \eqref{C14}, we see that the term in $s^5$ in the brackets reads
\[
-3 s^5\vf _x^2(\vf _x ^3)_x = -9s^5 \vf _x ^4 \vf _{xx} = -9s^5 \frac{(\psi ')^4 \psi ''}{t^5(T-t)^5}\cdot
\]
We infer from \eqref{C2} that for some $\kappa _1>0$ and all $s>0$ 
\[
 -9s^5 \vf _x ^4 \vf _{xx}\ge \kappa _1 (s\varphi )^5\qquad (t,x)\in (0,T)\times  ( [0,L]\setminus \omega ).
\]
On the other hand, we have for some $\kappa _2>0$ and all $s>0$
\begin{eqnarray*}
|\al _t| + |\al _{xxx}| + |9s^3 ( (\vf _x ^2\vf _{xx})_{xx} - (\vf _x\vf _{xx}^2)_x)| &\le& \kappa _2 s^3\vf ^4 \quad (t,x)\in (0,T)\times (0,L),\\
|3s^2 \vf _x ^2\alpha _x| &\le& \kappa_2 (s\vf  )^5 \quad (t,x)\in (0,T)\times \omega . 
\end{eqnarray*}
Claim 1 follows then for all $s>s_1$ with $s_1$ large enough and some $C_1>1$. \\
{\sc Claim 2.}
There exist some  constants $s_2>0$ and $C_2>1$ such that for all $s\ge s_2$, we have
\be
\ii [3\al _x -27 s^3 \vf _x ^2 \vf _{xx} +3s (\vf _{xt} + \vf _{4x}) ]u_x^2 \ge C_2^{-1} \ii (s\vf ) ^3 u_x ^2
- C_2\int_0^T\!\!\!\int_\omega (s\vf ) ^3 u_x ^2.
\ee
Indeed, the term in $s^3$ in the brackets is found to be 
\[
-18 s^3 \vf _x ^2 \vf _{xx} \ge \kappa _3 (s\vf ) ^3 \quad (t,x)\in (0,T)\times ([0,L]\setminus\omega )
\]
for some $\kappa _3 >0$ and all $s>0$, by \eqref{C2}. On the other hand, we have for some $\kappa _4>0$ and all $s>0$
\begin{eqnarray*}
|6s(\vf _{tx} + \vf _{4x})| &\le& \kappa _4s\vf ^2\quad (t,x)\in (0,T)\times (0,L) , \\
|18s^3\vf _x^2\vf _{xx}| &\le& \kappa_4 (s\vf ) ^3\quad (t,x)\in (0,T)\times \omega .
\end{eqnarray*}
Claim 2 follows for all $s\ge s_2$ with $s_2$ large enough and some $C_2>1$.\\
{\sc Claim 3.}
There exist some constants $s_3>0$ and $C_3>1$ such that for all $s\ge s_3$, we have
\be
 -9s \ii \vf _{xx}u_{xx}^2 \ge C_3^{-1} \ii s\vf  u_{xx}^2 - C_3 \int_0^T\!\!\!\int _\omega s\vf u_{xx}^2 .
\ee
Claim 3 is clear, for $\psi ''<0$ on $[0,L]\setminus \omega$.\\
{\sc Claim 4.}
There exist some constants $s_4>0$ and $C_4>1$ such that for all $s\ge s_4$, we have
\begin{multline*}
\int [ 3s \vf _x u _{xx}^2 + (9s^3 \vf _x^3 -3s \vf _{xxx} -\al ) u_x^2  + 2 \vf _{xx} u_x u_{xx} ]  \big\vert _0^L  \\
\ge C_4^{-1} \int_0^T [  (s\vf u^2_{xx} )_{\vert x=0} +  (s\vf u^2_{xx} )_{\vert x=L} +(s^3\vf ^3 u^2_{x} )_{\vert x=0}] dt.
\end{multline*}
Since ${u_x}_{\vert x=L}=0$  and 
\[
[( 9s^3\vf _x^3  -3s \vf _{xxx} -\alpha  ) u_x^2]_{\vert x=0} = 
[(8s^3\vf _x^3 -s (\vf _t + 4\vf _{xxx})) {u_x^2}]_{\vert x=0}, 
\]
we obtain with \eqref{C3} for $s\ge s_4$ with $s_4$ large enough, 
\[
[( 9s^3\vf _x^3   -3s \vf _{xxx} -\alpha  ) u_x^2] \big\vert _0^L  \ge \kappa_5 [(s\vf)^3u_x^2]_{\vert x=0}
\]
and 
\[
3s \vf _x u_{xx}^2 \vert _0^L \ge \kappa _6 ([s\vf u_{xx}^2]_{\vert x=0}  + [s\vf u_{xx}^2]_{\vert x=L}  ) 
\]
for some constant $\kappa_5,\kappa _6>0$. Finally
\[
| [ 2s\vf_{xx} u_x u_{xx} ]_{x=0} | \le \frac{\kappa _6}{2} [s\vf  u_{xx}^2]_{\vert x=0}  + \kappa _7 [s\vf u_x ^2]_{\vert x=0}
\]
for some constant $\kappa _7>0$. Since $s\vf (t,0) \ll (s\vf )^3 (t,0) $ for $s\gg 1$, Claim 4 follows. 

We infer from Claims 1, 2, 3, and 4 that for some positive constants $s_0,C$ and all $s\ge s_0$ 
\begin{multline}
\ii  [ (s\vf )^5 |u|^2 + (s\vf ) ^3 |u_x|^2 + s\vf |u_{xx}|^2 ]
+  \int_0^T [  (s\vf u^2_{xx} )_{\vert x=0} +  (s\vf u^2_{xx} )_{\vert x=L} +(s^3\vf ^3 u^2_{x} )_{\vert x=0}] dt \\
\le C ( \ii |w|^2 + \int_0^T \!\!\! \int_{\omega}  [ (s\vf )^5 |u|^2 + (s\vf ) ^3 |u_x|^2 + s\vf |u_{xx}|^2 ] \  ). \label{C40}
\end{multline}
Replacing $u$ by $e^{-s\vf }q$  yields \eqref{C7}.
\qed

\noindent
{\em Proof of Corollary \ref{cor11}.}  Note first that for $\xi\in Y_\frac{1}{4}$ and $v_T\in L^2(0,L)$, one can prove that \eqref{nc2} 
has a unique solution $v\in Y_\frac{1}{4}$, by using the contraction mapping principle for the integral equation. 
Corollary \ref{cor11} follows from Proposition \ref{prop10} by taking $q_0(x)=v_T(L-x)$,
$q(t,x)=v(T-t,L-x)$,  and  $f(t,x)=-\xi (T-t,L-x) q_x(t,x)$, assuming first that $\xi\in Y_\frac{1}{4}\cap L^\infty (Q)$ (so that $f\in L^2(Q)$). 
Indeed, with $u=e^{-s\vf }q$, 
\[
w=e^{-s\vf } L( e^{ s\vf} u) = -\xi (T-t,L-x) (u_x+ s\vf _x u), 
\] 
so that 
\ba
\int\!\!\!\int |w|^2 dx dt &\le&   C \int_0^T\!\!\! \int_0^L |\xi (T-t,L-x)|^2 (|u_x|^2 + |s\vf _x u|^2) dxdt \nonumber \\
& \le &  C   \int_0^T || \xi (T-t) ||^2_{L^2(0,L)} \big(  ||u_x||^2_{L^\infty (0,L)} + ||s\vf _x u||^2 _{L^\infty (0,L)} \big) dt \nonumber \\
& \le & C || \xi ||^2_{L^\infty (0,T,L^2(0,L) ) } \int_0^T\!\!\!\int_0^L [u_x^2+u_{xx}^2 + \frac{s^2}{t^2(T-t)^2} (u^2+u_x^2)]dx.\label{C41}
\ea
Combining \eqref{C40} with \eqref{C41}, picking $s\gg 1$, and replacing again $u$ by $e^{-s\vf} v(T-t,L-x)$ yields \eqref{D60}. 
The result for $\xi\in Y_\frac{1}{4}$ follows by density.\qed
\subsubsection{Internal observation\label{Sec3.1.2}}

We go back to the adjoint system (\ref{nc2}). Our next goal is to remove the terms $v_{xx}$ and $v_x$ from the r.h.s. of
\eqref{D60}. In addition to the weight $\tilde \varphi(  t,x)  =\frac{1}{t(  T-t)  }\psi( L- x)$, we introduce the functions
\begin{equation}
\hat{\varphi} (t) =  \frac{1}{t(  T-t)  } \max_{x\in [0,L]} \psi( x)
  =\frac{\psi(  0)  }{t(  T-t)  }\text{ and }
\check{\varphi} (t)=  \frac{1}{t(  T-t)  }  \min_{x\in [0,L]} 
\psi (x) =\frac{\psi(  l_{3})  }{t(  T-t)  },
\label{o5}%
\end{equation}
where we used \eqref{C4}. By \eqref{C5},  we have 
\begin{equation}
\hat{\varphi}(t)<\frac{4}{3}\check{\varphi}(t),\quad t\in (0,T). \label{o5_new}%
\end{equation}

\begin{lemma}
\label{newcarleman}Let $0<l_1<l_2<L$, $\xi\in Y_\frac{1}{4}$,
 and $\tilde s_{0}$ be as in Corollary \ref{cor11}. Then 
there exists a constant $C=C(T, || \xi ||_{ Y_{\frac{1}{4} } } ) >0$ such that for any $s\ge \tilde s_0$ and any $v_{T}\in
L^{2}(  0,L)  $, the solution $v$ of (\ref{nc2}) satisfies
\begin{equation}%
\begin{array}
[c]{lll}%
{\displaystyle\int_{Q}} \left\{  (s \check{\varphi} )^{5} |v|^{2}+ (s\check{\varphi}) ^{3}|v_{x}|^{2}+
 s\check{\varphi} |v_{xx}| ^{2}\right\}  e^{-2s\hat{\varphi}} dxdt & \leq & C_{1} s^{10} {\displaystyle\int_{0}^{T}}e^{s(
6\hat{\varphi}-8\check{\varphi})  }\check{\varphi}^{31}\left\Vert
v(  t,\cdot)  \right\Vert _{L^{2}(  \omega)  }^{2}dt,
\end{array}
\label{NewCarl}%
\end{equation}
where $Q=(0,T)\times(0,L)$ and $\omega=(l_{1},l_{2})\subset(0,L)$.
\end{lemma}

\begin{proof}
We follow the same approach as in \cite{GG}. From (\ref{D60}) and
(\ref{o5})-(\ref{o5_new}), we first obtain%
\begin{multline}%
\displaystyle\int_{Q}
\left\{ s^{5}\check{\varphi}^{5}  |v| ^{2}+ s^{3}\check{\varphi}%
^{3}|v_{x}|^{2}+s\check{\varphi} |v_{xx}|^{2}\right\}  e^{-2s\check{\varphi}} dxdt\\
\leq C \displaystyle \int_{0}^{T}\!\!\!\int_{\omega}\left\{
s^{5}\check{\varphi}^{5}\left\vert v\right\vert ^{2}+s^{3}\check{\varphi}%
^{3}|v_{x}|^{2}+s\check{\varphi}|v_{xx}|^{2}\right\}  e^{-2s\check{\varphi}}dxdt =:C(I_{0}+I_{1}+I_{2}).%
\label{L2_new}%
\end{multline}
Since $\check{\varphi}$ and $ \hat{\varphi}$ do not depend on $x$, we clearly have that 
\begin{equation}
I_{1}\leq s^3 \int_{0}^{T}\check{\varphi}^{3}e^{-2s\check{\varphi}}\left\Vert
v(  t,\cdot)  \right\Vert _{H^{1}(  \omega)  }^{2}dt
\label{o6}%
\end{equation}
and%
\begin{equation}
I_{2}\leq s\int_{0}^{T}\check{\varphi}e^{-2s\check{\varphi}}\left\Vert
v(  t,\cdot)  \right\Vert _{H^{2}(  \omega)  }%
^{2}dt\text{.} \label{o7}%
\end{equation}

Using interpolation in the Sobolev spaces $H^s(\omega)$ ($s\ge 0$),  we obtain for some positive constants $K_1,K_2$
\begin{equation}
\left\Vert v(  t,\cdot)  \right\Vert _{H^{1}(  \omega)
}\leq K_{1}\left\Vert v(  t,\cdot)  \right\Vert _{H^{8/3}(
\omega)  }^{3/8}\left\Vert v(  t,\cdot)  \right\Vert
_{L^{2}(  \omega)  }^{5/8} \label{o8}%
\end{equation} 
and
\begin{equation}
\left\Vert v(  t,\cdot)  \right\Vert _{H^{2}(  \omega)
}\leq K_{2}\left\Vert v(  t,\cdot)  \right\Vert _{H^{8/3}(
\omega)  }^{3/4}\left\Vert v(  t,\cdot)  \right\Vert
_{L^{2}(  \omega)  }^{1/4}\text{.} \label{o9}%
\end{equation}
Replacing (\ref{o8}) and (\ref{o9}) in (\ref{o6}) and (\ref{o7}),
respectively, yields%
\begin{equation}
I_{1}\leq Cs^3 \int_{0}^{T}\check{\varphi}^{3}e^{-2s\check{\varphi}}\left\Vert
v(  t,\cdot)  \right\Vert _{H^{8/3}(  \omega)  }%
^{3/4}\left\Vert v(  t,\cdot)  \right\Vert _{L^{2}(
\omega)  }^{5/4}dt \label{o10}%
\end{equation}
and%
\begin{equation}
I_{2}\leq Cs\int_{0}^{T}\check{\varphi}e^{-2s\check{\varphi}}\left\Vert
v(  t,\cdot)  \right\Vert _{H^{8/3}(  \omega)  }%
^{3/2}\left\Vert v(  t,\cdot)  \right\Vert _{L^{2}(
\omega)  }^{1/2}dt\text{.} \label{o11}%
\end{equation}
Next, an application of Young inequality in (\ref{o10}) and (\ref{o11}) gives%
\begin{align}
I_{1}    
&  \leq Cs^3\int_{0}^{T}\check{\varphi}^{3}e^{-2s\check{\varphi}}e^{-\frac{3}{4}%
s\hat{\varphi}}e^{\frac{3}{4}s\hat{\varphi}}\check{\varphi}^{-\frac{27}{8}%
}\check{\varphi}^{\frac{27}{8}}\left\Vert v(  t,\cdot)  \right\Vert
_{H^{8/3}(  \omega)  }^{3/4}\left\Vert v(  t,\cdot)
\right\Vert _{L^{2}(  \omega)  }^{5/4}dt\nonumber\\
&  \leq C_{\epsilon} s^6 \int_{0}^{T}e^{s(  \frac{6}{5}\hat{\varphi}-\frac
{16}{5}\check{\varphi})  }\check{\varphi}^{51/5}\left\Vert v(
t,\cdot)  \right\Vert _{L^{2}(  \omega)  }^{2}dt
 +\epsilon s ^{-2} \int_{0}^{T}e^{-2s\hat{\varphi}}\check{\varphi}^{-9}\left\Vert
v(  t,\cdot)  \right\Vert _{H^{8/3}(  \omega)  }%
^{2}dt\label{o12}
\end{align}
and%
\begin{align}
I_{2}  
&  \le Cs \int_{0}^{T}e^{-2s\check{\varphi}}e^{-\frac{3}{2}s\hat{\varphi}}%
e^{\frac{3}{2}s\hat{\varphi}}\check{\varphi}^{-\frac{27}{4}}\check{\varphi
}^{\frac{31}{4}}\left\Vert v(  t,\cdot)  \right\Vert _{H^{8/3}%
(  \omega)  }^{3/2}\left\Vert v(  t,\cdot)  \right\Vert
_{L^{2}(  \omega)  }^{1/2}dt\nonumber\\
&  \leq C_{\epsilon} s^{10} \int_{0}^{T}e^{s(  6\hat{\varphi}-8\check{\varphi
})  }\check{\varphi}^{31}\left\Vert v(  t,\cdot)  \right\Vert
_{L^{2}(  \omega)  }^{2}dt
 +\epsilon  s^{-2} \int_{0}^{T}e^{-2s\hat{\varphi}}\check{\varphi}^{-9}\left\Vert
v(  t,\cdot)  \right\Vert _{H^{8/3}(  \omega)  }%
^{2}dt\text{,}\label{o13}
\end{align}
for any $\epsilon>0$. Note that 
\be
\label{o13bis}
I_0 + 
s^6 \int_{0}^{T}e^{s(  \frac{6}{5}\hat{\varphi}-\frac{16}{5}\check{\varphi
})  }\check{\varphi}^{51/5}\left\Vert v(  t, \cdot )
\right\Vert _{L^{2}(  \omega)  }^{2}dt
\leq 
Cs^{10} {\displaystyle\int_{0}^{T}} e^{s(  6\hat{\varphi}-8\check{\varphi
})  }\check{\varphi} ^{31}\left\Vert v(  t,\cdot)
\right\Vert _{L^{2}(  \omega)  }^{2}dt.
\ee
Gathering together  (\ref{L2_new})   and (\ref{o12})-(\ref{o13bis}), 
we obtain%
\begin{multline}%
{\displaystyle\int_{Q}} 
\left\{ s^{5}\check{\varphi}^{5}  |v| ^{2}+ s^{3}\check{\varphi}%
^{3}|v_{x}|^{2}+s\check{\varphi} |v_{xx}|^{2}\right\}  e^{-2s\check{\varphi}} dxdt\\
 \leq  C s^{10} {\displaystyle\int_{0}^{T}} e^{s(
6\hat{\varphi}-8\check{\varphi})  }\check{\varphi} ^{31}\left\Vert
v(  t,\cdot)  \right\Vert _{L^{2}(  \omega)  }^{2}dt
+  2\epsilon s^{-2} {\displaystyle\int_{0}^{T}} e^{-2s\hat{\varphi}}\check{\varphi
}^{-9}\left\Vert v(  t,\cdot)  \right\Vert _{H^{8/3}(
\omega)  }^{2}dt\text{.}%
\label{o14}%
\end{multline}
It remains to estimate the integral term
\[
{\displaystyle\int_{0}^{T}} e^{-2s\hat{\varphi}}\check{\varphi}^{-9}\left\Vert
v(  t,\cdot)  \right\Vert _{H^{8/3}(  \omega)  }%
^{2}dt\text{.}
\]

Let $v_{1}(  t,x)  :=\theta_{1}(  t)  v(
t,x)  $ with
\[
\theta_{1}(  t)  =\exp(  -s\hat{\varphi})
\check{\varphi}^{-\frac{1}{2}}\text{.}%
\]
Then $v_{1}$ satisfies the system
\begin{equation}
\left\{
\begin{array}
[c]{lll}%
-v_{1t}-v_{1xxx}=f_{1}:=\xi\theta_{1}v_{x}-\theta_{1t}v &  & \text{in }(
0,T)  \times(  0,L)  \text{,}\\
v_{1}(  t,0)  =v_{1}(  t,L)  =v_{1x}(  t,0)
=0 &  & \text{in }(  0,T)  \text{,}\\
v_{1}(  T,x)  =0
&  & \text{in }(  0,L)  .
\end{array}
\right.  \label{o15}%
\end{equation}
Now, observe that, since $v_x(t,0)=0$, $\xi\in L^\infty (0,T,L^2(0,L)) $ and $\left\vert \theta_{1t}\right\vert
\leq C s \check{\varphi}^{\frac{3}{2}}\exp(  -s\hat{\varphi})  $, we have%
\ba
\left\Vert f_{1}\right\Vert _{L^{2}(  (  0,T)  \times( 0,L)  )  }^{2}
&\leq&  C ||\xi||^2_{L^\infty (0,T,L^2(0,L))}\int_0^T e^{-2s\hat \vf} ||v_x||^2_{L^\infty (0,L)} dt +
C\int_Q e^{-2s\hat{\varphi}}  s^2\check{\varphi}^{3}|v|^{2}dxdt  \nonumber\\
&\leq&
C \int _Q \left\{  s^2 \check\varphi ^{3}  |v| ^{2}+ s |v_{x}|^{2}+s^{-1}  |v_{xx}|^{2}\right\}  e^{-2s\check{\varphi}} dxdt \label{o16}
\ea
for some constant $C>0$ and all $s\ge s_0$.
Moreover, by Proposition \ref{est3}, $v_{1}\in Y_{1/2}$. Then, interpolating between 
$L^{2}(  0,T;H^{2}(  0,L)  )  $ and $L^{\infty}(
0,T;H^{1}(  0,L)  )  $, we infer that $v_{1}\in L^{4}(
0,T;H^{3/2}(  0,L)  )  $ and%
\begin{equation}
\left\Vert v_{1}\right\Vert _{L^{4}(  0,T;H^{3/2}(  0,L)
)  }\leq C\left\Vert f_{1}\right\Vert _{L^{2}(  (  0,T)
\times(  0,L)  )  }\text{.} \label{o17}%
\end{equation}

Let $v_{2}(  t,x)  :=\theta_{2}(  t)  v(
t,x)  $ with
\[
\theta_{2}=\exp(  -s\hat{\varphi})  \check{\varphi}^{-\frac{5}{2}%
}\text{.}%
\]
Then $v_{2}$ satisfies system (\ref{o15}) with $f_{1}$ replaced by%
\[
f_{2}:=\xi\theta_{2}\theta_{1}^{-1}v_{1x}-\theta_{2t}\theta_{1}^{-1}%
v_{1}\text{.}%
\]

Observe that%
\[
\left\vert \theta_{2}\theta_{1}^{-1}\right\vert +\left\vert \theta_{2t}%
\theta_{1}^{-1}\right\vert \leq Cs\text{.}%
\]
On the other hand, since $\xi\in L^4(  0,T;H^\frac{1}{2}(  0,L))  $  and $v_{1x}\in L^{4}(
0,T;H^{\frac{1}{2}}(  0,L)  )  $ by (\ref{o17}), we infer that
$\xi v_{1x}\in L^{2}(  0,T;H^{1/3}(  0,L)  )  $ (the product of two functions in $H^\frac{1}{2}(0,L)$ being in 
$H^\frac{1}{3}(0,L)$). Thus,
we obtain%
\begin{equation}
\left\Vert f_{2}\right\Vert _{L^{2}(  0,T;H^{1/3}(  0,L)
)  }\leq Cs \left\Vert v_{1}\right\Vert _{L^{4}(  0,T;H^{3/2}(
0,L)  )  }\text{.} \label{o19}%
\end{equation}

Interpolating between (\ref{16}) and (\ref{18}), we have that $v_{2}\in L^{2}(
0,T;H^{7/3}(  0,L)  )  \cap L^\infty (  0,T;H^{4/3}  (  0,L)  )  $ with
\begin{equation}
\left\Vert v_{2}\right\Vert _{L^{2}(  0,T;H^{7/3}(  0,L)
)  \cap L^{\infty}(  0,T;H^{4/3}(  0,L)  )  }\leq
C\left\Vert f_{2}\right\Vert _{L^{2}(  0,T;H^{1/3}(  0,L)
)  }\text{.} \label{o18}%
\end{equation}

Finally, let $v_{3}:=\theta_{3}(  t)  v(  t,x)  $ with%
\[
\theta_{3}(  t)  =\exp(  -s\hat{\varphi})
\check{\varphi}^{-\frac{9}{2}}\text{.}%
\]
Then $v_{3}$ satisfies system (\ref{o15}) with $f_{1}$ replaced by%
\[
f_{3}:=\xi\theta_{3}\theta_{2}^{-1}v_{2x}-\theta_{3t}\theta_{2}^{-1}%
v_{2}\text{.}%
\]
Again%
\[
\left\vert \theta_{3}\theta_{2}^{-1}\right\vert +\left\vert \theta_{3t}%
\theta_{2}^{-1}\right\vert \leq Cs\text{.}%
\]
Interpolating again between (\ref{16}) and (\ref{18}), we have that%
\begin{equation}
\left\Vert v_{3}\right\Vert _{L^{2}(  0,T;H^{8/3}(  0,L)
)  \cap L^{\infty}(  0,T;H^{5/3}(  0,L)  )  }\leq
C\left\Vert f_{3}\right\Vert _{L^{2}(  0,T;H^{2/3}(  0,L) )  }.
 \label{o20}%
\end{equation}
Since $\xi\in Y_\frac{1}{4}$, we have that $\xi\in L^3(  0,T;H^\frac{2}{3}(  0,L))$. On the other hand, by \eqref{o18},
\[
v_{2x} \in  L^{2}(  0,T;H^{4/3}(  0,L) )  \cap L^{\infty}(  0,T;H^{1/3}(  0,L)  ).
\]
It follows that $v_{2x}\in L^6(0,T,H^\frac{2}{3}(0,L))$. 
Since $H^\frac{2}{3}(0,L)$ is an algebra, we  conclude that 
$\xi v_{2x} \in L^2(0,T, H^\frac{2}{3}(0,L))$. 
Therefore
\be
\left\Vert f_{3}\right\Vert _{L^{2}(  0,T;H^{2/3}(  0,L) )  }
\leq  C s \left\Vert v_{2}\right\Vert _{L^{2}(  0,T;H^{7/3}(  0,L)
)  \cap L^{\infty}(  0,T;H^{4/3}(  0,L)  )  }.\label{o21}
\ee
Thus we infer from (\ref{o16})-(\ref{o21}) that for some constants $C_1,C_2>0$ and all $s\ge s_0$%
\begin{eqnarray}
\left\Vert v_{3}\right\Vert ^2_{L^{2}(  0,T;H^{8/3}(  0,L) )  }
&\leq& C_1s^4 || f_1||^2_{ L^2( (0,T)\times (0,L) ) } \nonumber \\
&\leq& C_{2} {\displaystyle\int_{Q}}  
\left\{ s^{6}\check{\varphi}^{3}  |v| ^{2}+ 
s^{5} |v_{x}|^{2}+s^3  |v_{xx}|^{2}\right\}  e^{-2s\check{\varphi}} dxdt.\label{o22}%
\end{eqnarray}
Hence, replacing $v_3 =\exp(  -s\hat{\varphi}) \check{\varphi}^{-\frac{9}{2}} v$ in \eqref{o22} yields for some constant $C_3>0$ 
\begin{equation}
{\displaystyle\int_{0}^{T}} e^{-2s\hat{\varphi}}\check{\varphi}^{-9}\left\Vert
v(  t,\cdot)  \right\Vert _{H^{8/3}(  \omega)  } ^{2}dt
\leq  C_{3} s^2 {\displaystyle\int_{Q}} 
\left\{ (s\check{\varphi} )^{5}  |v| ^{2}+ (s\check{\varphi} )^3 |v_{x}|^{2}+s\check{\varphi}  |v_{xx}|^{2}\right\}  e^{-2s\check{\varphi}} dxdt.\\
\label{o23}%
\end{equation}
Then, picking $\epsilon= 1/(4C_3)$ in (\ref{o14}) results in 
\[%
\begin{array}
[c]{lll}%
{\displaystyle\int_{Q}} s\check{\varphi}e^{-2s\hat{\varphi}}\left\{
s^{4}\check{\varphi}^{4} |v|^{2}+s^{2}\check{\varphi}^{2}|v_{x}|^{2}%
+|v_{xx}|^{2}\right\}  dxdt & \leq & C_4 s^{10} {\displaystyle\int_{0}^{T}}
e^{s(  6\hat{\varphi}-8\check{\varphi})  }\check{\varphi}
^{31}\left\Vert v(  t,\cdot)  \right\Vert _{L^{2}(
\omega)  }^{2}dt%
\end{array}
\]
for  all $s\ge \tilde s_0$ and some positive constant $C_4=C_4(T, ||\xi ||_{Y_{ \frac{1}{4} } } )$. 
\end{proof}

We are in a position to prove the null controllability
of system (\ref{nc1}).

\begin{theorem}
\label{nullsyslin}Let $T>0$. Then there exists $\delta >0$ such that for any  $\xi\in Y_{1/4}$ with $||\xi ||_{L^2(0,T,H^1(0,L))} \le\delta $ and any 
$u_{0}\in L^{2}(0,L)  $, one may find a control 
$f\in L^{2}(  (  0,T)  \times\omega)  $ such that the solution
$u$ of (\ref{nc1}) fulfills $u(  T,\cdot)  =0.$
\end{theorem}

\begin{proof}
Scaling in \eqref{nc2} by $v$ and $(L-x)v$, we obtain after some computations
the estimate
\[
\label{abc}
||v||^2_{L^\infty(0,T,L^2(0,L)) } + 2 ||v_x||^2_{L^2(0,T,L^2(0,L))} \le C(L) \left( || v_T ||^2_{L^2(0,L)}
+ ||\xi ||^2_{L^2(0,T,H^1(0,L)) } ||v_x||^2_{L^2(0,T,L^2(0,L)) } \right)
\]
for some constant $C(L)>0$. 
It follows that if $||\xi||_{L^2(0,T,H^1(0,L))}\le \delta := 1/\sqrt{C(L)}$, then we have 
\begin{equation}
\label{abc1}
\max_{t\in [0,T]} ||v (t)  ||^2_{L^2(0,L) } +  ||v_x||^2_{L^2(0,T,L^2(0,L))} \le C(L)  || v_T ||^2_{L^2(0,L)}.
\end{equation}
Replacing $v(t)$ by $v(0)$ and $v_T$ by $v(\tau )$ for $T/3<\tau <2T/3$ in \eqref{abc1}, and integrating over $\tau \in (T/3,2T/3)$, we obtain that
\be
\label{abc2}
||v(0)||^2_{L^2(0,L)} \le \frac{3C(L)}{T}  \int_{\frac{T}{3}}^{\frac{2T}{3}} ||v(\tau )||^2_{L^2(0,L)} d\tau. 
\ee
Combining \eqref{abc2} with Lemma \ref{newcarleman} for a fixed value of $s\ge \tilde s_0$,  we derive the following observability inequality%
\begin{equation}
\int_{0}^{L}\left\vert v(  0,x)  \right\vert ^{2}dx\leq C_{\ast
}\int_{0}^{T}\left\Vert v(  t,\cdot)  \right\Vert _{L^{2}(
\omega)  }^{2}dt \label{o24}%
\end{equation}
where $C_{\ast}=C_{\ast}(T, ||\xi||_{ Y_{1/4} } ) >0$. 
Using (\ref{o24}), we can deduce the existence of a function $v\in
L^{2}(  (  0,T)  \times\omega)  $ as in Theorem \ref{nullsyslin}
proceeding as follows.

On $L^{2}(  0,L)$, we define the norm
\[
\left\Vert v_{T}\right\Vert _{B}:=\left\Vert v\right\Vert _{L^{2}(
(  0,T)  \times\omega)  }\text{,}%
\]
where $v$ is the solution of (\ref{nc2}) associated with $v_{T}$. The fact that  $||\cdot ||_B$ 
is a norm comes from (\ref{o24}) applied on $(t,T)$ for $0<t<T$.

Let $B$ denote the completion of $L^{2}(  0,L)  $ with respect to the above
norm. We define a functional $J$ on $B$ by%
\[
J(  v_{T})  :=\frac{1}{2}\left\Vert v_{T}\right\Vert _{B}^{2}%
+\int_{0}^{L}v(  0,x)  u_{0}(  x)  dx\text{.}%
\]
From (\ref{o24}) we infer that $J$ is well defined and continuous on $B$. As
it is strictly convex and coercive, it admits a unique minimum $v_{T}^{\ast}%
$, characterized by the Euler-Lagrange equation%
\begin{equation}
\int_{0}^{T}\!\!\!\int_{\omega}v^{\ast}wdxdt+\int_{0}^{L}w( 0,x)
u_{0}(  x)  dx=0\text{,} \qquad \forall w_T\in B,\label{o25}%
\end{equation}
where $w$ (resp. $v^{\ast}$) denotes the solution of (\ref{nc2}) associated
with $w_{T}\in B$ (resp. $v_{T}^{\ast}\in B $). Define $f$ $\in L^{2}( (  0,T)  \times\omega)  $ by%
\begin{equation}
f:=1_{\omega}v^{\ast} , \label{o26}%
\end{equation}
and let $u$ denote the solution of (\ref{nc1}) associated with $u_{0}$  and $f$. 
Multiplying (\ref{nc1}) by $w(  t,x)  $ and integrating by parts,
we obtain for all $w_T\in L^2(0,L)$%
\begin{equation}
\int_{0}^{L}u(  T,x)  w_{T}dx=\int_{0}^{L} u_0(x)w(  0,x)  dx +\int_{0}^{T}\!\!\!\int_{\omega} v^*w
dxdt=0\text{,} \label{o27}%
\end{equation}
where the second equality follows from (\ref{o25}). Therefore $u(  T,\cdot)  =0$.
Finally, letting $w_T=v_T^*$ in \eqref{o25} and using \eqref{o24}, we obtain 
\be
\label{o40}
\int_0^T\!\!\!\int_{\omega} | f |^2dxdt \le C_\ast \int_0^L |u_0(x)|^2 dx. 
\ee

\end{proof}

\subsection{Null controllability of the nonlinear equation}

In this section we prove Theorem \ref{nullsysnon}. This is done by using
a fixed-point argument.

\subsubsection{Proof of Theorem \ref{nullsysnon}}

Consider $u$ and $\bar{u}$ fulfilling system (\ref{ncn2}) and (\ref{ncn1}),
respectively. Then $q=u-\bar{u}$ satisfies%
\begin{equation}
\left\{
\begin{array}
[c]{lll}%
q_{t}+q_{x}+( \frac{q^2}{2} +  \bar{u}q)_x  +q_{xxx}=1_{\omega}
f( t,x)  &  & \text{in }(  0,T)  \times(  0,L)
\text{,}\\
q(  t,0)  =q(  t,L)  =q_{x}(  t,L)  =0 &  &
\text{in }(  0,T)  \text{,}\\
q(  0,x)  =q_{0}(  x)  :=u_{0}(  x)  -\bar
{u}_{0}(  x)  &  & \text{in }(  0,L)  \text{.}%
\end{array}
\right.  \label{ncn3}%
\end{equation}
The objective is to find $f$ such that the solution $q$ of (\ref{ncn3}) satisfies%
\[
q(  T,\cdot)  =0\text{.}%
\]

Given $\xi\in Y_\frac{1}{4}$ and $q_0:=u_0-\bar{u}_0\in L^2(0,L)$, we consider the control problem
\ba
q_{t}+q_{x}+(\xi q)_x  +q_{xxx}=1_{\omega}
f(t,x)  &  & \text{in }(  0,T)  \times(  0,L)
\text{,} \label{CP1}\\
q(  t,0)  =q(  t,L)  =q_{x}(  t,L)  =0 &  &
\text{in }(  0,T)  \text{,} \label{CP2}\\
q(  0,x)  =q_{0}(  x)   &  & \text{in }(  0,L)  \text{.} \label{CP3}%
\ea
We can prove the following estimate
\begin{multline}
\label{estimCP}
||q||^2_{L^\infty(0,T,L^2(0,L))} + 2 ||q_x||^2_{L^2(0,T,L^2(0,L))} 
\le \tilde C(L)\big( ||q_0||^2_{L^2(0,L)} \\
+ ||\xi ||^2_{L^2(0,T,H^1(0,L))}  ||q_x ||^2_{ L^2(0,T,L^2(0,L))}  + || f || ^2_{L^2( (0,T)\times \omega )} \big)
\end{multline}
Let $\tilde \delta =\min (\delta , 1/\sqrt{\tilde C(L)})$. 
We introduce the space%
\[
E:=C^0(  [0,T];L^2(  0,L)  )  \cap L^2(
0,T;H^1(  0,L)  )  \cap H^{1}(  0,T;H^{-2}(
0,L)  )
\]
endowed with its natural norm 
\[
\left\Vert  z \right\Vert _{E} := ||z||_{Y_{1/4}} + || z ||_{ H^1 ( 0,T, H^{-2} (0,L) ) } .
\]
We consider in $L^{2}(  (  0,T)  \times(  0,L)  )  $ the
following set%
\[
B:=\left\{  z\in E; \ \left\Vert z\right\Vert _{E} \le 1 \  \textrm { and } \   ||z||_{L^2(0,T,H^1(0,L))} \le \tilde \delta  \right\} \text{.}%
\]
$B$ is compact in $L^2((0,T)\times (0,L))$, by Aubin-Lions' lemma. We will limit ourselves to controls $f$ fulfilling the
condition 
\begin{equation}
\label{bound}
|| f ||^2_{L^2((0,T)\times \omega )} \le C_* ||q_0||^2_{L^2(0,L) } 
\end{equation} 
where $C_\ast := C_\ast (T, || \bar u ||_{Y_ {1/4} }  + \frac{1}{2} )$. 
We associate with any $z\in B$ the set%
\[
\begin{array}
[c]{c}%
T(  z)  :=\left\{  q\in B;\ \ \exists f\in L^{2}(  (0,T)  \times\omega)  \text{ such that }f \text{ satisfies } \eqref{bound}  \text{  and }\right. \\
\left.  q\text{ solves \eqref{CP1}-\eqref{CP3} with }\xi=  \bar{u} +\frac{z}{2}
\text{ and }q(  T,\cdot)  =0 \right\} .
\end{array}
\]
Note that $||\bar{u}||_{L^2(0,T,H^1(0,L))}<\tilde\delta/2$ for $T\ll1$. 
By Theorem \ref{nullsyslin} and (\ref{estimCP}), we see that if $\left\Vert
q_0\right\Vert _{L^2(  0,L)  }$  and $T$ are sufficiently small, then
$T(  z)  $ is nonempty for all $z\in B$. We shall use the
following version of Kakutani fixed point theorem (see e.g. \cite[Theorem
9.B]{Zeidler}):

\begin{theorem}
\label{fixed}Let $F$ be a locally convex space, let $B\subset F$ and let
$T:B\longrightarrow2^{B}$. Assume that

\begin{enumerate}
\item $B$ is a nonempty, compact, convex set;

\item $T(  z)  $ is a nonempty, closed, convex set for all $z\in B$;

\item The set-valued map $T:B\longrightarrow2^{B}$ is upper-semicontinuous;
i.e., for every closed subset $A$ of $F$, $T^{-1}(  A)  =\left\{
z\in B;\ T(  z)  \cap A\neq\varnothing\right\}  $ is closed.
\end{enumerate}

Then $T$ has a fixed point, i.e., there exists $z\in B$ such that $z\in
T(  z)  $.
\end{theorem}

Let us check that Theorem \ref{fixed} can be applied to $T$ and%
\[
F=L^{2}(  (  0,T)  \times(  0,L)  )  \text{.}%
\]
The convexity of $B$ and $T(  z)  $ for all $z\in B$ is clear. Thus
(1) is satisfied. For (2), it remains to check that $T(  z)  $ is
closed in $F$ for all $z\in B$. Pick any $z\in B$ and a sequence $\left\{
q^{k}\right\}  _{k\in\mathbb{N}}$ in $T(z)$ which converges in $F$ towards some
function $q\in B$. For each $k$, we can pick some control function $f^{k}\in
L^{2}(  (  0,T)  \times\omega)  $ fulfilling \eqref{bound} 
such that \eqref{CP1}-\eqref{CP3} are satisfied with $\xi= \bar{u} +\frac{z}{2}  $
and $q^{k}(  T,\cdot)  =0$. Extracting subsequences if needed, we
may assume that as $k\rightarrow\infty$%
\ba
f^{k}\rightarrow f &  & \text{in }L^{2}(  (  0,T)
\times\omega)  \text{ weakly,}  \label{t1} \\
q^k\rightarrow q &  & \text{in }L^{2}(  0,T;H^{1}(
0,L)  )  \cap H^{1}(  0,T;H^{-2}(  0,L)  )
\text{ weakly,}\label{t2}
\ea
By \eqref{t2}, the boundedness of $|| q^k ||_{ L^\infty (0,T,L^2(0,L))}$  and Aubin-Lions' lemma, $\{ q^k\}_{k\in \mathbb N}$ is relatively compact in $C^0([0,T],H^{-1}(0,L))$. 
Extracting a subsequence if needed, we may assume that
\[
q^k\rightarrow q \text{ strongly in } C^0([0,T],H^{-1}(0,L)). 
\] 
In particular, $q(0,x)=q_0(x)$ and $q(T,x)=0$.  On the other hand, we infer from \eqref{t2} that 
\[
\xi q^{k}\rightarrow\xi q\text{ in }L^{2}(  (  0,T)
\times(  0,L)  )  \text{ weakly.}%
\]
Therefore, $(\xi q^{k})_x \rightarrow (\xi q)_x$ in  ${\mathcal D}'(  (  0,T) \times(  0,L))$.   
Finally,  it is clear that 
\[ || f ||^2_{L^2((0,T)\times \omega )} \le C_* ||q_0||^2_{L^2(0,L)} \]
and that $q$ satisfies \eqref{CP1} with $\xi=  \bar{u} +\frac{z}{2}$ and $q(T,\cdot)  =0$.
Thus $q\in T(  z)  $ and  $T(  z)  $ is
closed. Now, let us check (3). To prove that $T$ is upper-semicontinuous, consider
any closed subset $A$ of $F$ and any sequence $\left\{  z^{k}\right\}
_{k\in\mathbb{N}}$ in $B$ such that%
\begin{equation}
z^{k}\in T^{-1}(  A) , \quad \forall k\ge 0,  \label{ncn6}%
\end{equation}
and
\begin{equation}
z^{k}\rightarrow z\ \text{ in }\ F \label{ncn5}%
\end{equation}
for some $z\in B$. We aim to prove that $z\in T^{-1}(  A)  $. By
(\ref{ncn6}), we can pick a sequence $\left\{  q^{k}\right\}  _{k\in\mathbb{N}}$
in $B$ with $q^k\in T(z^k)\cap  A$  for all $k$, and a sequence $\left\{  f^{k}\right\}  _{k\in\mathbb{N}}$ in
$L^{2}(  (  0,T)  \times\omega)  $ such that%
\begin{equation}
\left\{
\begin{array}
[c]{lll}%
q_{t}^{k}+ q_x^k +  (  (  \bar{u}+\dfrac{z^{k}}{2})  q^{k})_{x}+q_{xxx}^{k}=1_{\omega}f^{k}(  t,x)  &  & \text{in }(
0,T)  \times(  0,L)  \text{,}\\
q^{k}(  t,0)  =q^{k}(  t,L)  =q_{x}^{k}( t,L)  =0 &  & \text{in }(  0,T)  \text{,}\\
q^{k}(  0,x)  =q_{0}(  x)  &  & \text{in }( 0,L)  ,
\end{array}
\right.  \label{ncn7}%
\end{equation}%
\begin{equation}
q^{k}(  T,x)  =0, \qquad \text{ in } (0,L), \label{ncn8}%
\end{equation}
and%
\begin{equation}
\left\Vert f^{k}\right\Vert ^2_{L^2(  (  0,T)  \times
\omega)  }\leq C_*\left\Vert q_{0}\right\Vert _{L^2(  0,L)} ^2. \label{ncn9}%
\end{equation}
From (\ref{ncn9}) and the fact that $z^{k}$, $q^{k}\in B$, extracting
subsequences if needed, we may assume that as $k\rightarrow\infty$,%
\[%
\begin{array}
[c]{lll}%
f^{k}\rightarrow f &  & \text{in }L^{2}(  (  0,T)
\times\omega)  \text{ weakly,}\\
q^{k}\rightarrow q &  & \text{in }L^{2}(  0,T;H^{1}( 0,L)  )  \cap H^{1}(  0,T;H^{-2}(  0,L)  )
\text{ weakly,}\\
q^{k}\rightarrow q & & \text{in } C^0([0,T],H^{-1}(0,L)) \text{ strongly},\\
q^{k}\rightarrow q &  & \text{in }F\text{ strongly,}\\
z^{k}\rightarrow z &  & \text{in }F\text{ strongly,}\\
\end{array}
\]
where $f\in L^2((0,T)\times \omega )$ and $q\in B$. 
Again, $q(0,x)=q_0(x)$ and $q(T,x)=0$.  We also see that \eqref{CP2} and \eqref{bound} are satisfied. It remains to check that%
\begin{equation}
q_{t} +q_x +(  (  \bar{u} +\frac{z}{2})  q)  _{x}+q_{xxx}=1_{\omega } f (  t,x)  \text{.} \label{ncn10}%
\end{equation}
Observe that the only nontrivial convergence in (\ref{ncn7}) is those of the nonlinear term
$(  z^{k}q^{k})  _{x}$. Note first that 
\[
||z^kq^k||_{L^2(0,T,L^2(0,L))} \le ||z^k || _{L^\infty (0,T,L^2(0,L))} ||q^k||_{L^2(0,T,L^\infty(0,L))} \le C,
\]
so that, extracting a subsequence, one can assume that $z^kq^k\rightarrow f$ weakly in $L^2((0,T)\times (0,L))$. 
To prove that $f=zq$, it is sufficient to observe that
for any $\varphi\in {\mathcal D} (Q)$, 
\[
\int_0^T\!\!\!\int_0^L z^kq^k\varphi dxdt\to \int_0^T\!\!\!\int_0^L z q\varphi dxdt,
\]
for $z^k\to z$ and $q^k\varphi \to q\varphi$ in $F$. 
Thus%
\[%
\begin{array}
[c]{lll}%
z^{k}q^{k}\rightarrow zq &  & \text{in }L^{2}(  (  0,T)
\times(  0,L)  )  \text{ weakly.}%
\end{array}
\]
It follows that $(  z^{k}q^{k})  _{x}\rightarrow( zq)  _{x}$ in $\mathcal{D}^{\prime}(  (  0,T)
\times(  0,L)  )  $. Therefore, (\ref{ncn10}) holds and 
$q\in T(  z)  $. On the other hand, $q\in A$, since $q^{k}\rightarrow q$ in $F$ and $A$ is closed. We conclude that $z\in T^{-1}(  A)  $,
and hence $T^{-1}(  A)  $ is closed.

Il follows from Theorem \ref{fixed} that there exists
$q\in  B$ with $q\in T(  q )$, i.e. we have found a control 
$f\in L^{2}((  0,T)  \times\omega)  $ such that the solution of
(\ref{ncn3}) satisfies $q(  T,\cdot)  =0$ in $(  0,L)$. The proof of Theorem \ref{nullsysnon} is complete.\qed

With Theorem \ref{nullsysnon} at hand, one can prove Theorem \ref{thmC} about the regional controllability.  

\subsection{Proof of Theorem \ref{thmC}.}
By Theorem \ref{nullsysnon}, if $\delta$ is small enough one can find a control input $f\in L^2(0,T/2,L^2(0,L))$ with 
$\text{supp} (f) \subset (0,T)\times \omega$ such that the solution of 
\eqref{S1} satisfies $u(T/2,.)\equiv 0$ in $(0,L)$. Pick any number $l_2'\in (l_1',l_2)$ with $l_2'\not\in {\mathcal N}$. 
(This is possible,  the set $\mathcal N$ being discrete.)
By \cite[Theorem 1.3]{Rosier}, if $\delta$ is small enough 
one can pick a function $h\in L^2(T/2,T)$ such that the solution $y\in C^0([T/2,T],L^2(0,l_2'))\cap L^2(T/2,T,H^1(0,l_2'))$ of the system
\[
\left\{
\begin{array}{ll}
y_t+y_{xxx}+y_x +yy_x =0 \quad &\text{ in } (T/2,T)\times (0,l_2'),\\
y(t,0)=y(t,l_2')=0,\ \ y_x(t,l_2')=h(t) &\text{ in } (T/2,T),\\ 
y(T/2,x)=0 &\text{ in } (0,l_2')
\end{array}
\right.
\]
satisfies $y(T,x)=u_1(x)$ for $0<x<l_2'$. We pick a function $\mu\in C^\infty ( [0,L] ) $ such that 
\[
\mu (x) =\left\{ 
\begin{array}{ll}
1\quad &\text{ if } x<l_1',\\
0&\text{ if } x>\frac{l_1'+l_2'}{2}
\end{array}
\right. 
\]
and set for $T/2< t \le T$
\[
u(t,x)=
\left\{ 
\begin{array}{ll}
\mu(x) y(t,x)\quad &\text{ if } x<l_2',\\
0&\text{ if } x>l_2'.
\end{array}
\right.  
\]
Note that, for $T/2<t<T$, $u_t+u_{xxx}+u_x+uu_x=f$ with 
\[
f=\mu(\mu -1) yy_x+(\mu_{xxx} y + 3\mu_{xx}y_x + 3\mu _x y_{xx} + \mu _x y) + \mu \mu _x y^2.
\]
Since $||y||^4_{ L^4(0,T,L^4(0,l_2'))}\le C||y||^2_{L^\infty(0,T,L^2(0,L))} ||y||^2_{L^2(0,T,H^1(0,L))}$, it is clear that $f\in L^2(0,T,H^{-1}(0,L))$ with 
$\text{supp} (f)\subset (0,T)\times (l_1,l_2)$. Furthermore, 
$u\in C([0,T],L^2(0,L))\cap L^2(0,T,H^1(0,L))$ solves \eqref{S1} and satisfies \eqref{corS1}. \qed

\section{Exact controllability results\label{Sec3}}
\label{section4}
Pick any function $\rho \in C^\infty (0,L)$ with 
\be
\label{H4}
\rho (x)= \left\{
\begin{array}{ll}
0 \quad &\textrm{ if } \ 0<x<L-\nu ,\\
1 &\textrm{ if }\  L-\frac{\nu}{2}  <x<L,
\end{array} 
\right.  
\ee
for some $\nu \in (0,L)$.

This section is devoted to the investigation of  the exact controllability
of the system%
\begin{equation}
\left\{
\begin{array}
[c]{lll}%
u_{t}+u_{x}+uu_{x}+u_{xxx}= f = (\rho (x)h)_x  &  & \text{in }(
0,T)  \times(  0,L)  \text{,}\\
u(  t,0)  =u(  t,L)  =u_{x}(  t,L)  =0 &  &
\text{in }(  0,T)  \text{,}\\
u(  0,x)  =u_{0}(  x)  &  & \text{in }(
0,L)  \text{.}%
\end{array}
\right.  \label{ec1}%
\end{equation}
More precisely, we aim to find a control input $h\in L^{2}(  0,T;L^{2}( 0,L)  )  $ (actually, with $(\rho (x)h(t,x))_x$ in some space of functions) to 
guide the system described by (\ref{ec1}) in the time interval $[0,T]$ from any
(small) given initial state $u_{0}$ in $L^2_{\frac{1}{L-x} dx}$ to any (small) given terminal state $u_{T}$ in the same space. We first consider the linearized system,
and next proceed to the nonlinear one. The results involve some weighted Sobolev spaces.

\subsection{The linear system}
For any measurable function $w:(0,L)\to (0,+\infty )$ (not necessarily in $L^1(0,L)$), we introduce the weighted $L^2-$space
\[
L^2_{w(x)dx} =\{ u\in L^1_{loc}(0,L);\ \int_0^L u(x)^2 w(x)dx <\infty \}. 
\] 
It is a Hilbert space when endowed with the scalar product
\[
(u,v)_{L^2_{w(x)dx}}=\int_0^L u(x)v(x) w(x)dx. 
\]
We first prove the well-posedness of the linear system associated with
(\ref{ec1}), namely%
\begin{equation}
\left\{
\begin{array}
[c]{lll}%
u_{t}+u_{x}+u_{xxx}=0 &  & \text{in }(  0,T)  \times(
0,L)  \text{,}\\
u(  t,0)  =u(  t,L)  =u_{x}(  t,L)  =0 &  &
\text{in }(  0,T)  \text{,}\\
u(  0,x)  =u_{0}(  x)  &  & \text{in }(
0,L)  \text{,}%
\end{array}
\right.  \label{ec_lin}%
\end{equation}
in both  the spaces $L^2_{xdx}$ and $L^2_{\frac{1}{L-x}dx}$, following  \cite{goubet} where the well-posedness
was established in $L^2_{\frac{x}{L-x}dx}$. We need 
the following result.
\begin{theorem}
\label{goubet} 
(see \cite{goubet})
Let $W\subset V\subset H$ be three Hilbert spaces with continuous and dense embeddings. 
Let $a(v,w)$ be a bilinear form defined on $V\times W$ that satisfies the following properties:\\
(i) ({\bf Continuity})
\begin{equation}
\label{gou1}
a(v,w) \le M ||v||_V ||w||_W, \quad \forall v\in V,\ \forall w\in W;
\end{equation}
(ii) ({\bf Coercivity})
\begin{equation}
\label{gou2}
a(w,w) \ge m ||w||_V^2, \quad \forall w\in W;
\end{equation}
Then for all $f\in V'$ (the dual space of $V$), there exists $v\in V$ such that 
\begin{equation}
\label{gou3}
a(v,w)=f(w), \qquad \forall w\in W.
\end{equation}
If, in addition to (i) and (ii), $a(v,w)$ satisfies\\
(iii) ({\bf Regularity})
for all $g\in H$, any solution $v\in V$ of \eqref{gou3} with $f(w):=(g,w)_H$ belongs to $W$,\\ then 
\eqref{gou3} has a unique solution $v\in W$. Let $D(A)$ denote the set of those $v\in W$ when 
$g$ ranges over $H$, and set $Av= - g$. Then $A$ is a maximal dissipative operator, and hence it generates  a continuous semigroup of contractions
$(e^{tA})_{t\ge 0}$ in $H$.  
\end{theorem}
\subsection{Well-posedness in $L^2_{xdx}$}
\begin{theorem}
Let $A_1u=-u_{xxx}-u_x$ with domain 
\[ D(A_1)=\{ u\in H^2(0,L)\cap H^1_0(0,L); \ u_{xxx}\in L^2_{xdx}, \ u_x(L)=0\} \subset L^2_{xdx} .\]
 Then $A_1$ generates a
strongly continuous semigroup in $L^2_{xdx}$.   
\end{theorem}
\begin{proof}
Let 
\[
H=L^2_{xdx}, \quad V=H^1_0(0,L),\quad W=\{ w\in H^1_0 (0,L),\ w_{xx}\in L^2_{x^2dx} \}, 
\]
be endowed with the respective norms
\[
||u||_H : = ||\sqrt{x} u||_{L^2(0,L)},\quad ||v||_V := ||v_x||_{L^2(0,L)}, \quad ||w||_W := ||xw_{xx}||_{L^2(0,L)}. 
\]
Clearly, $V\subset H$ with a continuous (dense) embedding between two Hilbert spaces. On the other hand, we have that 
\be
\label{Q1}
||w_x||_{L^2} \le  C ||xw_{xx}||_{L^2} \qquad \forall w\in W. 
\ee  
First, we note that we have for $w\in {\mathcal T} :=C^\infty( [0,L] )\cap H^1_0(0,L)$ and $p\in \R$
\[
0\le \int_0^L (xw_{xx}+pw_x) ^2dx = \int_0^L(x^2w^2_{xx} + 2px w_xw_{xx} + p^2 w_x^2) dx = \int_0^L x^2w_{xx}^2 dx + (p^2-p) \int_0^L w_x^2 dx + pLw_x^2(L).
\] 
Taking $p=1/2$ results in 
\begin{equation}
\label{poi}
\int_0^L w_x^2 dx \le 4\int_0^L x^2 w_{xx}^2 dx + 2L |w_x(L)|^2.
\end{equation} 
The estimate \eqref{poi} is also true for any $w\in W$, since $\mathcal T$ is dense in $W$.  Let us prove 
\eqref{Q1} by contradiction. If \eqref{Q1} is false, then there exists a sequence $\{ w^n \}_{n\ge 0}$ in $W$ such that 
\[
1 = ||w_x^n||_{L^2} \ge n ||x w^n_{xx}||_{L^2}\qquad \forall n\ge 0. 
\]
Extracting subsequences, we may assume  that
\begin{eqnarray*}
w^n&\to& w\quad \text{ in } H^1_0(0,L) \text{ weakly}\\
xw^n_{xx} &\to& 0\quad \text{ in } L^2(0,L) \text{ strongly}
\end{eqnarray*}
and hence $xw_{xx}=0$, which gives $w(x)=c_1x+c_2$. Since $w\in H^1_0(0,L)$, we infer that $w\equiv 0 $. 
Since $w^n$ is bounded in $H^2(L/2,L)$, extracting subsequences we may also assume that $w_x^n (L)$ converges in $\mathbb R$. 
We infer then from \eqref{poi} that $w^n$ is a Cauchy sequence  in $H^1_0(0,L)$, so that 
\[
w^n \to w\quad \text{ in } H^1_0(0,L) \text{ strongly},
\] 
and hence $||w_x||_{L^2} =\lim_{n\to\infty} ||w_x^n||_{L^2} =1$. This contradicts the fact that $w\equiv 0$. The proof of \eqref{Q1} is achieved. 

Thus $||\cdot ||_W$ is a norm in $W$, which is clearly a Hilbert space, and $W\subset V$ with continuous (dense) embedding. Let 
\[
a(v,w) = \int_0^L v_x [(xw)_{xx} + xw ] dx, \qquad v\in V, \ w\in W.
\] 
Let us check that (i), (ii), and (iii) in Theorem \ref{goubet} hold. For $v\in V$ and $w\in W$, 
\begin{eqnarray*}
|a(v,w)| &\le& 
||v_x||_{L^2} ||xw_{xx} + 2w_x + xw ||_{L^2} \\
&\le& ||v_x||_{L^2} \big( ||xw_{xx}||_{L^2} + C||w_x||_{L^2}\big) \\
&\le& C ||v||_V ||w||_W
\end{eqnarray*}
where we used Poincar\'e inequality and \eqref{Q1}. This proves that the bilinear form $a$ is well defined and continuous on $V\times W$. 
For (ii), we first pick any $w\in {\mathcal T}$ to obtain
\begin{eqnarray*}
a(w,w) &=& \int_0^L w_x ( xw_{xx} + 2w_x +xw) dx \\
&=& \frac{3}{2} \int_0^L w_x^2 dx + [x\frac{w_x^2}{2}]\vert _0^L -\frac{1}{2} \int_0^L w^2 dx\\
&\ge & \frac{3}{2} \int_0^L w_x^2 dx - \frac{1}{2} \int_0^L w^2 dx.
\end{eqnarray*}
By Poincar\'e inequality
\[
\int_0^L w^2(x) dx \le ( \frac{L}{\pi} )^2 \int_0^L w_x^2(x)dx,
\]
and hence
\[
a(w,w) \ge (\frac{3}{2} -\frac{L^2}{2\pi ^2})\int_0^L w_x^2 dx.  
\]
This shows the coercivity when $L< \pi \sqrt{3}$. When $L\ge \pi \sqrt{3}$, we have to consider, instead of $a$, the bilinear form 
$a_\lambda(v,w):=a(v,w) + \lambda (v,w)_H$ for $\lambda \gg 1$. Indeed, we have by Cauchy-Schwarz inequality and Hardy inequality
\begin{eqnarray*}
||w||^2_{L^2} &\le & ||x^{\frac{1}{2}} w||_{L^2} ||x^{-\frac{1}{2}} w||_{L^2} \\
&\le& \sqrt{L} ||w||_{H} ||x^{-1} w||_{L^2} \\
&\le& \varepsilon ||w_x||^2_{L^2} +C_\varepsilon ||w||^2_H
\end{eqnarray*}  
and hence
\[
a_\lambda (w,w) \ge (\frac{3}{2} -\frac{\varepsilon}{2}) ||w||^2_V + (\lambda -\frac{C_\varepsilon }{2}) ||w||^2_H.
\]
Therefore, if $\varepsilon <3$ and $\lambda>C_\varepsilon/2$, then $a_\lambda$ is a continuous bilinear form which is coercive. 

Let us have a look at the regularity issue. For given $g\in H$, let $v\in V$ be such that 
\[
a_\lambda(v,w) =(g,w)_H \qquad \forall w\in W,  
\]
i.e.
\be
\label{Q3}
\int_0^L v_x ((xw)_{xx} + xw)dx + \lambda \int_0^L v(x)w(x) x dx = \int_0^L g(x) w(x) xdx.
\ee
Picking any $w\in {\mathcal D}(0,L)$ results in 
\be
\label{Q33} 
\langle x(v_{xxx}+ v_x + \lambda v), w\rangle _{{\mathcal D}',{\mathcal D}}
=\langle xg, w\rangle _{{\mathcal D}',{\mathcal D}}
\qquad \forall w\in {\mathcal D} (0,L),  
\ee
and hence
\be
\label{Q2} 
v_{xxx} +v_x +\lambda v = g \qquad \text{ in } {\mathcal D} '(0,L). 
\ee
Since $v\in H^1_0(0,L)$ and $g\in L^2_{xdx}$, we have that $v\in H^3(\varepsilon , L)$ for all $\varepsilon \in (0,L)$ and 
$v_{xxx}\in L^2_{xdx}$. Picking any $w \in {\mathcal T}$ and $\varepsilon \in (0,L)$, and scaling in \eqref{Q2} by $xw$ yields
\[
\int_\varepsilon^L v_x((xw)_{xx} +xw)dx + [v_{xx} (xw) -v_x (xw)_x]\vert _\varepsilon ^L = \int_\varepsilon^L (g-\lambda v)xw dx. 
\]
Letting $\varepsilon \to 0$ and comparing with \eqref{Q3}, we obtain
\be
\label{Q4}
-Lv_x(L)w_x(L) =
\lim_{\varepsilon \to 0} \big( \varepsilon v_{xx}(\varepsilon ) w(\varepsilon ) -v_x (\varepsilon ) (w(\varepsilon ) + \varepsilon w_x(\varepsilon )) \big) .
\ee
Since $v_{xxx}\in L^2_{xdx}$, we obtain successively for some constant $C>0$ and all $\varepsilon \in (0,L)$ 
\ba
|v_{xx}(\varepsilon ) -v_{xx} (L) | &\le & (\int_\varepsilon ^L x  | v_{xxx} |^2 dx )^\frac{1}{2} (\int_\varepsilon ^L x^{-1}dx)^\frac{1}{2} \le C |\log \varepsilon | 
\label{Q10}\\ 
|v_x(\varepsilon) | &\le& C.   \label{Q11}
\ea 
We infer from \eqref{Q10} that $v\in H^2(0,L)$, and hence $v\in W$. Furthermore, letting $\varepsilon \to 0$ in \eqref{Q4} and using \eqref{Q10}-\eqref{Q11} yields
$v_x(L)=0$, since $w_x(L)$ was arbitrary. We conclude that $v\in {\mathcal D} (A_1)$. Conversely, it is clear that the operator 
$A_1-\lambda $ maps ${\mathcal D} (A_1)$ into $H$, and actually onto $H$ from the above computations. 
Hence $A_1-\lambda $ generates a strongly semigroup of contractions in $H$.  
\end{proof}
\subsection{Well-posedness in $L^2_{(L-x)^{-1}dx}$}
\begin{theorem} 
\label{thm30}
Let $A_2u=-u_{xxx}-u_x$ with domain 
\[ 
{\mathcal D}(A_2)=\{ u\in H^3(0,L)\cap H^1_0(0,L); \ u_{xxx}\in L^2_{ \frac{1}{L-x} dx} \text{ and }  \ u_x(L)=0 \} \subset L^2_{ \frac{1}{L-x} dx} .
\]
 Then $A_2$ generates a
strongly continuous semigroup in $L^2_{\frac{1}{L-x} dx}$.   
\end{theorem}
\begin{proof}
We will use Hille-Yosida theorem, and (partially) Theorem \ref{goubet}. Let 
\be
\label{AAA1}
H=L^2_{ \frac{1}{L-x} dx}, \quad V=\{u\in H^1_0(0,L), \ u_x\in L^2_{\frac{1}{ (L-x)^2} dx} \}, \quad W=H^2_0(0,L),  
\ee
be endowed respectively with the norms 
\be
\label{AAA2}
||u||_H = || (L-x)^{-\frac{1}{2}} u||_{L^2}, \quad ||u||_V  = || (L-x)^{-1} u_x ||_{L^2}, \quad || u ||_W = ||u_{xx}||_{L^2}.  
\ee
From \cite{goubet}, we know that $V$ endowed with $||\cdot||_V$ is a Hilbert space, and that 
\be
||(L-x)^{-2} u||_{L^2} \le \frac{2}{3} || (L-x)^{-1} u_x ||_{L^2} \qquad \forall u\in V,
\label{P1}
\ee
and hence
\be
\label{P2}
||u||_H \le (\int_0^L \frac{L^3}{ (L-x)^4 }u^2(x)dx )^\frac{1}{2} \le \frac{2}{3} L^\frac{3}{2} ||u||_V \qquad \forall u\in V. 
\ee
Thus $V\subset H$ with continuous embedding. From Poincar\'e inequality, we have that $|| \cdot ||_W$ is a norm
on $W$ equivalent to the $H^2-$norm. On the other hand, from Hardy inequality
\be
\label{P2P}
\int_0^L \frac{v^2}{(L-x)^2} dx \le C \int_0^L v_x ^2 dx \quad \forall v\in H^1(0,L)\text{ with } v(L)=0,
\ee   
we have that 
\begin{equation}
\label{P2P2}
||v||_V \le C||v||_W\qquad \forall v \in W.
\end{equation}
Thus $W\subset V$ with continuous embedding. It is easily  seen that ${\mathcal D} (0,L)$ is dense in $H$, $V$, and $W$. Let 
\[
a(v,w) = \int_0^L [v_x ( \frac{w}{L-x} )_{xx} + v_x\frac{w}{L-x} ]dx\qquad (v,w)\in V\times W. 
\] 
Then
\begin{eqnarray*}
|a(v,w)| 
&\le &     \vert \int_0^L v_x (\frac{w_{xx} }{L-x} + 2\frac{w_x}{ (L-x)^2 }  + 2\frac{w}{ (L-x)^3} +\frac{w}{L-x} )dx \vert \\
&\le & ||w_{xx}||_{L^2} ||\frac{v_x}{L-x}||_{L^2} + 2|| \frac{w_x}{L-x} ||_{L^2} || \frac{v_x}{L-x} ||_{L^2} 
+|| \frac{v_x}{L-x} ||_{L^2} \big( 2  ||\frac{w}{ (L-x)^2} ||_{L^2} + ||w||_{L^2} \big) \\
&\le& C ||v||_V ||w||_W
\end{eqnarray*}
by \eqref{P1}, \eqref{P2}, and \eqref{P2P2}. 
This shows that $a$ is well defined and continuous. Let us look at the coercivity of $a$. Pick any $w\in {\mathcal D} (0,L)$. Then 
\begin{eqnarray*}
a(w,w) &=& \int_0^L w_x \big( \frac{ w_{xx} }{L-x} + 2\frac{w_x}{ (L-x)^2 } + 2\frac{w}{(L-x)^3} + \frac{w}{L-x}  \big) dx  \\
&=& \frac{3}{2} \int_0^L \frac{w^2_x}{(L-x)^2} dx -3 \int_0^L \frac{w^2}{(L-x)^4} dx - \frac{1}{2} \int_0^L \frac{w^2}{(L-x)^2} dx\\
&\ge&  \frac{1}{6}\int_0^L \frac{w_x^2}{ (L-x)^2 } dx -\frac{1}{2} \int_0^L \frac{w^2 }{ (L-x)^2} dx
\end{eqnarray*} 
where we used \eqref{P1} for the last line. Note that, using Cauchy-Schwarz inequality and \eqref{P1}, we have that 
\ba
||\frac{w}{L-x}||^2_{L^2} 
&\le& || (L-x)^{-\frac{1}{2}}  w||_{L^2}   || (L-x)^{ -\frac{3}{2}} w ||_{L^2} \nonumber \\
&\le& \frac{2 \sqrt{L} }{3} || w ||_H || w ||_V \nonumber \\
&\le & \varepsilon ||w||_V^2 + \frac{L}{9\varepsilon } ||w||_H^2.
\ea 
If we pick $\varepsilon \in (0,1/3)$, we infer that for all $w \in {\mathcal D} (0,L)$ 
\be
a(w,w) + \frac{L}{18\varepsilon } || w ||^2_H \ge \big( \frac{1}{6} - \frac{\varepsilon}{2} \big)  || w ||^2_V \ge C || w||^2_V.  
\label{AAA10}
\ee
The result is also true for any $w\in W$, by density. This shows that the continuous bilinear form 
\[
a_\lambda (v,w) = a(v,w) + \lambda (v,w)_H
\]  
is coercive for $\lambda > L/6$. Let $g\in H$ be given. By Theorem \ref{goubet}, there is at least one solution $v\in V$ of 
\be
\label{P8}
a_\lambda (v,w) =(g,w)_H \qquad \forall w\in W. 
\ee  
Pick such a solution $v\in V$, and let us prove that $v\in {\mathcal D} (A_2)$. Picking any $w\in {\mathcal D} (0,L)$ in \eqref{P8} yields
\be
\label{P10}
v_{xxx} + v_x + \lambda v=g \qquad \text{ in } {\mathcal D}'(0,L). 
\ee 
As $g\in L^2(0,L)$ and $v\in H^1(0,L)$, we have that $v_{xxx}\in L^2(0,L)$, and $v\in H^3(0,L)$. Pick finally $w$ of the form 
$w(x)=x^2 (L-x)^2 \overline{w} (x)$, where 
$\overline{w}\in C^\infty ([0,L])$ is arbitrary chosen. Note that $w\in W$ and that $w/(L-x) \in H^1_0(0,L)\cap C^\infty ( [0,L] )$. Multiplying in 
\eqref{P10} by $w/(L-x)$ and integrating over $(0,L)$, we obtain after comparing with \eqref{P8}  
\[
0=-v_x (\frac{w}{L-x})_x \vert _0^L = -v_x \big( (2xL-3x^2)
\overline{w} + x^2 (L-x)\overline{w}_x \big)\vert _0^L = v_x(L)L^2\overline{w}(L). \]
As $\overline{w} (L)$ can be chosen arbitrarily, we conclude that $v_x(L)=0$. Using \eqref{P2P} twice, we infer that $v_x+\lambda v\in H$, and hence 
$v_{xxx}=g - (v_x+\lambda v)\in H$. 
Therefore
 $v\in {\mathcal D} (A_2)$. Thus, for $\lambda >L/6$ we have that $A_2-\lambda: {\mathcal D} (A_2) \to H$ is onto.
Let us check that $A_2-\lambda$ is also dissipative in $H$. Pick any $w\in {\mathcal D} (A_2)$. Then we obtain after some integrations by parts that 
 \[
 (A_2 w , w)_H = -\frac{3}{2} \int_0^L \frac{w_x^2}{ (L-x)^2 } dx + 3 \int_0^L \frac{w^2}{ (L-x)^4 } dx +\frac{1}{2} \int_0^L \frac{ w^2 }{ (L-x)^2 } dx 
 -\frac{ w_x^2(0) }{2L}
 \] 
 and
 \[
 (A_2 w-\lambda w, w)_H \le -(\frac{1}{6} -\frac{\varepsilon}{2}) ||w||_V^2 -\frac{w_x^2(0)}{2L} \le 0
 \]
 for $\varepsilon <1/3$ and $\lambda =L/(18\varepsilon)$.  We conclude that $A_2-\lambda $ is maximal dissipative for 
 $\lambda >L/6$, and thus it generates a strongly continuous semigroup of contractions in  $H$ by Hille-Yosida theorem. 
\end{proof}
A global Kato smoothing effect as in \cite{goubet,Rosier} can as well be derived.
\begin{proposition}
\label{prop40}
Let $H$ and $V$ be as in \eqref{AAA1}-\eqref{AAA2}, and let $T>0$ be given. Then there exists some constant $C=C(L,T)$ such that for any $u_0\in H$, the solution 
$u(t)=e^{tA_2}u_0$ of \eqref{ec_lin} satisfies
\be
\label{AAA3}
||u||_{L^\infty(0,T,H)} + ||u||_{L^2(0,T,V)} \le C ||u_0||_H. 
\ee 
\end{proposition}
\begin{proof}
We proceed as in \cite{goubet}. First, we notice that ${\mathcal D}(A_2)$ is dense in $H$, so that it is sufficient to prove the result when $u_0\in {\mathcal D}(A_2)$.
Note that the estimate  $||u||_{L^\infty(0,T,H)} \le C ||u_0||_H$ is a consequence of classical semigroup theory.
 Assume $u_0\in {\mathcal D}(A_2)$, so that $u_t=A_2u$ in the classical sense. Taking the inner product in $H$ with $u$ yields 
\[ (u_t,u)_H=-a(u,u) \le -C ||u||_V^2 + \frac{L}{18\varepsilon} ||u||_H^2 \]
where we used \eqref{AAA10}. An integration over $(0,T)$ completes the proof of the estimate of $||u||_{ L^2(0,T,V) } $. 
\end{proof}
 
\subsection{Non-homogeneous system}
In this section we consider the nonhomogeneous system
\ba
u_{t}+u_{x}+u_{xxx}=f(  t,x)  &  & \text{in }(  0,T)
\times(  0,L), \label{H5}\\
u(  t,0)  =u(  t,L)  =u_{x}(  t,L)  =0 &  &
\text{in }(  0,T),\label{H6}\\
u(  0,x)  =u_0 &  & \text{in }(  0,L). \label{H7}%
\ea
We need the prove the existence of a ``reasonable'' solution when solely $f\in L^2(0,T,H^{-1}(0,L))$. 

\begin{proposition}
\label{prop20}
Let $u_0\in L^2_{xdx}$ and $f\in L^{2}(  0,T;H^{-1}(  0,L)  )  $. Then there exists a unique solution $u\in C([0,T],L^2_{xdx} )\cap L^2(0,T,H^1(0,L))$ 
to \eqref{H5}-\eqref{H7}. Furthermore, there is some constant $C>0$ such that 
\be
\label{H700}
||u||_{L^\infty (0,T,L^2_{xdx})} + ||u||_{L^2(0,T,H^1(0,L))} \le C\big( ||u_0||_{L^2_{xdx}} + || f ||_{ L^2 (0,T,H^{-1}(0,L) } \big) . 
\ee
\end{proposition}
\begin{proof}
Assume first that $u_0\in {\mathcal D}(A_1)$ and $f\in C^0([0,T],{\mathcal D} (A_1))$ to legitimate the following computations. Multiplying each term in \eqref{H5} by 
$xu$ and integrating over $(0,\tau ) \times (0,L)$ where $0<\tau <T$ yields
\be
\frac{1}{2}\int_0^L x|u(\tau ,x)|^2 dx -\frac{1}{2}\int_0^L x|u_0(x)|^2 dx +\frac{3}{2}\int_0^\tau \!\!\!\int_0^L |u_x|^2 dxdt 
-\frac{1}{2}\int_0^\tau \!\!\!\int_0^L |u|^2 dx dt = \int_0^\tau\!\!\!\int_0^L xuf dxdt. \label{H8} 
\ee
$\langle .,. \rangle _{H^{-1},H^1_0}$ denoting the duality pairing between $H^{-1}(0,L)$ and $H^1_0(0,L)$, we have that for all $\varepsilon >0$ 
\be
\label{H9}
\int_0^\tau\!\!\!\int_0^L xuf  dxdt  = \int_0^\tau \langle f, xu\rangle _{ H^{-1},H^1_0 } \le \frac{\varepsilon}{2} \int_0^\tau \!\!\!\int_0^L u_x^2 dxdt + 
C_\varepsilon \int_0^\tau || f ||^2_{H^{-1}} dt. 
\ee  
The last term in the l.h.s. of \eqref{H8} is decomposed as
\[
\frac{1}{2}\int_0^\tau \!\!\!\int_0^L |u|^2 dxdt = \frac{1}{2}\int_0^\tau\!\!\!\int_0^{\sqrt{\varepsilon}}  |u|^2 dxdt  + 
\frac{1}{2}\int_0^\tau \!\!\! \int_{\sqrt{\varepsilon}}^L |u|^2 dxdt =:I_1 + I_2.
\]
We claim that 
\ba
I_1 &\le & \frac{\varepsilon}{2} \int_0^\tau \!\!\!\int_0^L |u_x|^2 dxdt, \label{H11}\\
I_2 &\le &  \frac{1}{2\sqrt{\varepsilon} } \int_0^\tau \!\!\!\int_0^L  x |u|^2 dxdt. \label{H12}
\ea
For \eqref{H11}, since $u(0,t)=0$ we have that for $(t,x)\in (0,T)\times (0, \sqrt{\varepsilon} )$ 
\[
|u(x,t)|\le \int_0^{\sqrt{\varepsilon}} |u_x| dx \le \varepsilon ^\frac{1}{4} \big( \int_0^{\sqrt{\varepsilon}}  |u_x|^2 dx\big) ^\frac{1}{2} 
\]
and hence 
\[
\int_0^{\sqrt{\varepsilon}} |u|^2 dx \le \varepsilon \int_0^{\sqrt{\varepsilon}} |u_x|^2 dx
\]
which gives \eqref{H11} after integrating over $t\in (0,\tau )$. \eqref{H12} is obvious.

Gathering together \eqref{H8}-\eqref{H12}, we obtain 
\begin{multline*}
\frac{1}{2} \int_0^L x|u(\tau , x)|^2 dx  + (\frac{3}{2}-\varepsilon ) \int_0^\tau \!\!\!\int_0^L |u_x|^2 dxdt \\
\le \frac{1}{2} \int_0^L x|u_0(x)|^2 dx + \frac{1}{2\sqrt{\varepsilon}} \int _0^\tau\!\!\!\int_0^L x|u|^2 dxdt + C_\varepsilon \int_0^\tau || f ||^2_{H^{-1}}dt.
\end{multline*}
Letting $\varepsilon = 1$ and applying Gronwall's lemma, we obtain 
\[
||u||^2_{L^\infty (0,T,L^2_{xdx}) } + ||u_x||^2_{L^2(0,T,L^2(0,L)) } \le 
C(T) \big( ||u_0||^2_{L^2_{xdx}} +||f||^2_{L^2(0,T,H^{-1}(0,L))} \big)  .
\]
This gives \eqref{H700} for $u_0\in D(A_1)$ and $f\in C^0([0,T],D(A_1))$. A density argument allows us to construct a solution
$u\in C([0,T],L^2_{xdx}) \cap L^2(0,T,H^1(0,L))$ of \eqref{H5}-\eqref{H7} satisfying \eqref{H700}  for $u_0\in L^2_{xdx}$ and $f\in L^2(0,T,H^{-1}(0,L))$.
The uniqueness follows from classical semigroup theory.  
\end{proof}

Our goal now is to obtain a similar result in the spaces $H$ and $V$ introduced in \eqref{AAA1}-\eqref{AAA2}. To do that, we limit ourselves to the situation 
when $f=(\rho (x)h)_x$ with $h\in L^2(0,T,L^2(0,L))$.  

\begin{proposition}
\label{prop45} 
Let $u_0\in H$ and $h\in L^2(0,T,L^2(0,L))$, and set $f:=(\rho (x)h)_x$.  Then there exists a unique solution $u\in C([0,T],H)\cap L^2(0,T,V)$ 
to \eqref{H5}-\eqref{H7}. Furthermore, there is some constant $C>0$ such that
\be
\label{BBB}
||u||_{L^\infty (0,T,H)} + ||u||_{L^2(0,T,V ) } \le C\big( ||u_0||_{H} + || h ||_{ L^2 (0,T,L^2(0,L))  } \big) . 
\ee
\end{proposition}

\begin{proof}
Assume that $u_0\in {\mathcal D}(A_2)$ and  $h\in C_0^\infty ((0,T)\times (0,L))$, so that $f\in C^1([0,T],H)$.  
Taking the inner product of $u_t-A_2u-f=0$ with $u$ in $H$ yields 
\be
\label{AAA90}
 (u_t,u)_H=-a(u,u)+(f,u)_H \le -C ||u||_V^2 + \frac{L}{18\varepsilon} ||u||_H^2 +(f,u)_H,
 \ee
where we used \eqref{AAA10}. 
Then
\begin{eqnarray*}
| (f,u)_H| &=& \vert \int_0^L (\rho (x)h)_x \frac{u}{L-x} dx \vert \\
&=& \vert \int_0^L \rho (x) h \big( \frac{u_{x}}{L-x} + \frac{u}{ (L-x)^2 } \big) dx \vert \\ 
&\le&  C  ||h||_{L^2} ( || \frac{u_{x}}{L-x} ||_{L^2} + || \frac{u}{ (L-x)^2 } ||_{L^2} )  \\
&\le& C||h||_{L^2} ||u||_V,
\end{eqnarray*}
where we used \eqref{P1} in the last line. 
Thus, we have that 
\[
|(f ,u)_H |  \le \frac{C}{2} || u ||_V^2 + C' || h ||_{L^2}^2
\]
which, when combined with \eqref{AAA90}, gives after integration over $(0,\tau)$ for $0<\tau<T$
\[
||u(\tau )||_{H} ^2 + C\int_0^\tau ||u||^2_V  dt \le  ||u_0||_H^2 + C'' \big( \int_0^\tau ||u||_H^2 dt + \int_0^\tau \!\!\!\int_0^L |h|^2 dxdt \big). 
\]
 An application of Gronwall's lemma yields
 \eqref{BBB} for  $u_0\in {\mathcal D} (A_2)$ and $h\in C_0^\infty ((0,T)\times (0,L))$. 
A density argument allows us to construct a solution
$u\in C([0,T],H) \cap L^2(0,T,V)$ of \eqref{H5}-\eqref{H7} satisfying \eqref{BBB}  for $u_0\in H$ and $h\in L^2(0,T,L^2(0,L))$.
The uniqueness follows from classical semigroup theory.  
\end{proof}

\subsection{Controllability of the linearized system}
We turn our attention to the control properties of the linear system
\ba
u_t+u_{xxx}+u_x &=& f=(\rho (x)h)_x, \label{H21}\\
u(t,0) = u(t,L) = u_x(t,L) &=& 0,  \label{H22}\\
u(0,x) &=& u_0(x). \label{H23}
\ea
\begin{theorem}
\label{thm11}
Let $T>0$ , $\nu \in (0,L)$ and $\rho (x)$ as in \eqref{H4}. Then there exists a continuous linear operator
$\Gamma : L^2_{\frac{1}{L-x} dx}\to L^2(0,T,L^2(0,L))\cap L^2_{(T-t)dt} (0,T,H^1(0,L))$ such that for any $u_1\in L^2_{\frac{1}{L-x} dx}$, the solution 
$u$  of \eqref{H21}-\eqref{H23} with  $u_0=0$ and $h=\Gamma (u_1)$ satisfies $u(T,x)=u_1(x)$ in $(0,L)$. 
\end{theorem}
Note that the forcing term $f=(\rho (x)h)_x$ is actually a {\em function} in $L^2_{(T-t) dt}(0,T,L^2(0,L))$ supported in $(0,T)\times (L-\nu ,L)$.  
\begin{proof}
We use the Hilbert Uniqueness Method  (see e.g. \cite{Lions1}). 
Introduce the adjoint system
\ba
-v_t-v_{xxx}-v_x &=& 0, \label{H25}\\
v(t,0) = v(t,L) = v_x(t,0) &=& 0,  \label{H26}\\
v(T,x) &=& v_T(x). \label{H27}
\ea
If $u_0\equiv 0$, $v_T\in {\mathcal D} (0,L)$, and $h\in {\mathcal D} ((0,T)\times (0,L))$, then multiplying in \eqref{H21} by $v$ and integrating over
$(0,T)\times (0,L)$  gives
\[
\int_0^L u(T,x)v_T(x)dx =\int_0^T\!\!\!\int_0^L (\rho (x)h)_x v dxdt =-\int_0^T\!\!\!\int_0^L \rho (x)h v_x dxdt.
\]
The usual change of variables $x\to L -  x$, $t\to T-t$, combined with Proposition \ref{prop20}, gives
\[
||v||_{L^\infty (0,T,L^2_{ (L-x)dx})} + || v || _{L^2(0,T,H^1 (0,L))} \le C ||v_T ||_{L^2_{ (L-x)dx} }.
\] 
By a limiting argument, we obtain that for all $h\in L^2(0,T,L^2(0,L))$ and all $v_T\in L^2_{(L-x)dx}$,
\[
\langle u(T,.), v_T\rangle _{L^2_{ \frac{1}{L-x} dx }, L^2_{(L-x)dx}}
=-\int_0^T (h,\rho (x)v_x)_{L^2}dt,
\]
where  $u$ and $v$  denote the solutions of \eqref{H21}-\eqref{H23} and \eqref{H25}-\eqref{H27}, respectively, and 
$\left\langle \cdot,\cdot\right\rangle _{L^2_{ \frac{1}{L-x} dx}, L^2_{(L-x)dx}}$ denotes the duality pairing between
$L^2_{\frac{1}{L-x} dx}$ and $L^2_{( L-x)  dx}$. We have to prove the following observability inequality
\begin{equation}
\label{observabilite}
||v_T||^2_{L^2_{ (L-x)dx}} \le C \int_0^T\!\!\!\int_0^L |\rho (x)v_x|^2dxdt
\end{equation}
or, equivalently, letting $w(t,x)=v(T-t,L-x)$, 
\be
\label{H30}
||w_0||^2 _{L^2_{xdx}} \le C \int_0^T\!\!\!\int_0^L |\rho (L-x)w_x|^2 dxdt
\ee
where $w$ solves
\be
\label{H40}
\left\{ 
\begin{array}{l}
w_t+w_{xxx}+w_x=0,\\
w(t,0)=w(t,L)=w_x(t,L)=0,\\
w(0,x)=w_0(x).
\end{array}
\right.
\ee
From \cite{Rosier}, we know that for any $q\in C^\infty ([0,T]\times [0,L])$ 
\begin{multline*}
-\int_0^T\!\!\! \int_0^L (q_t + q_{xxx}+q_x) \frac{w^2}{2} dxdt + \int_0^L (q\frac{w^2}{2})(T,x)dx - \int_0^L (q\frac{w^2}{2})(0,x)dx \\
+\frac{3}{2}\int_0^T\!\!\!\int_0^L q_x w_x^2 dxdt + \int_0^T (q\frac{w_x^2}{2})(t,0)dt=0.
\end{multline*}
We pick $q(t,x)=(T-t)b(x)$, where $b\in C^\infty ([0,L]) $ is nondecreasing and satisfies
\[
b(x) = \left\{ 
\begin{array}{ll}
x & \text{ if }\  0<x<\nu /4 ,\\
1 & \text{ if }\  \nu /2 <x<L. 
\end{array}
\right.
\]
This yields
\ba
||w_0||^2_{L^2_{xdx}}  
&\le& C(L,\nu ) \int_0^L b(x) w_0^2 (x) dx \nonumber \\
&\le& C(T,L,\nu ) \big(  \int_0^T\!\!\! \int_0^{\frac{\nu}{2}}  w_x^2 dxdt + \int_0^T\!\!\!\int_0^L w^2 dxdt\big) .\label{H45}
\ea
If the estimate 
\be
\label{H48}
||w_0||^2_{L^2_{xdx}} \le C \int_0^T\!\!\!\int_0^{\frac{\nu}{2}} w_x^2 dxdt
\ee
fails, then one can find a sequence $\{w_0^n\}\subset L^2_{xdx}$ such that 
\be
\label{H50}
1=||w_0^n||^2_{L^2_{xdx}} >n \int_0^T\!\!\!\int_0^{\frac{\nu}{2}} |w_x^n|^2 dxdt,
\ee
where $w^n$ denotes the solution of \eqref{H40} with $w_0$ replaced by $w_0^n$.  By \eqref{H700} and \eqref{H50}, $\{ w^n\} $ is bounded 
in $L^2(0,T,H^1(0,L))$, hence also in $H^1(0,T,H^{-2}(0,L))$ by \eqref{H40}. Extracting a subsequence, we have by Aubin-Lions' lemma  that
$w^n$ converges strongly in $L^2(0,T,L^2(0,L))$. Thus, using \eqref{H45} and \eqref{H50}, we see that $w_0^n$ is a Cauchy sequence in $L^2_{xdx}$,
and hence it converges strongly in this space. Let $w_0$ denote its limit in $L^2_{xdx}$, and let $w$ denote the corresponding solution of \eqref{H40}.
Then 
\begin{eqnarray*}
&&||w_0||_{L^2_{xdx}} = 1, \\
&&w^n \to w \qquad \text{ in } L^2(0,T,H^1(0,L)).  
\end{eqnarray*}  
But $w_x^n\to 0$ in $L^2(0,T,L^2(0,\nu /2))$ by \eqref{H50}. Thus $w_x\equiv 0$ in $(0,T)\times (0,\nu/2 )$, and hence 
$w(t,x) = g(t)$ (for some function $g$) in $(0,T)\times (0,\nu /2 )$. Since $w$ satisfies \eqref{H40}, we infer from $w(t,0)=0$ that $w\equiv 0$ in $(0,T)\times (0,\nu/2 )$, and also 
in $(0,T)\times (0,L)$  by Holmgren's theorem. This would imply that $w(0,x)=0$, in contradiction with $|| w_0||_{L^2_{xdx}}=1$.  Therefore \eqref{H48} is proved, 
and \eqref{H30} follows at once. 

We are in a position to apply H.U.M. Let $\Lambda (v_T) = (L-x)^{-1} u(T,.)\in L^2_{ (L-x) dx}$, where $u$ solves \eqref{H21}-\eqref{H23} with 
$h=-\rho (x) v_x$. Then $\Lambda :  L^2_{ (L-x) dx} \to L^2_{ (L-x) dx}$ is clearly continuous. On the other hand, from
\eqref{observabilite}
\[
\big( \Lambda (v_T), v_T\big)_{L^2_{ (L-x)dx }} = \langle u(T,.), v_T\rangle _{L^2_{ \frac{1}{L-x} dx }, L^2_{(L-x)dx}} 
=\int_0^T ||\rho (x)v_x||^2_{L^2} dt \ge C ||v_T||^2_{L^2_{(L-x) dx}} ,
\]
and it follows that the map $v_T \to \Lambda (v_T)$ is invertible in $L^2_{ (L-x)dx } $. 

Define the map $\Gamma: \ L^2_{\frac{1}{L-x} dx} \to L^2(0,T,L^2(0,L))$ by
$\Gamma (u_1) = h := -\rho (x)v_x$, 
where $v$ is the solution of \eqref{H25}-\eqref{H27} with $v_T=\Lambda ^{-1}( (L-x)^{-1}u_1)$. $\Gamma$ is continuous
from   $L^2_{\frac{1}{L-x} dx}$  to  $L^2(0,T,L^2(0,L))$, and the solution $u$ of \eqref{H21}-\eqref{H23} with $u_0=0$ and $h=\Gamma (u_1)$ satisfies $u(T,.)=u_1$. 
To prove that $\Gamma $  is also continuous from  $L^2_{\frac{1}{L-x}dx}$  into  
$L^2_{ (T - t ) dt} (0,T,H^1(0,L))$, it is sufficient to prove the following
estimate 
\[
\int_0^T ||v (t) ||^2_{H^2}  (T-t) dt \le C ||v_T||^2_{L^2_{(L-x)dx}},
\]
for the solutions of \eqref{H25}-\eqref{H27} or, alternatively, the estimate 
\be
\label{H60}
\int_0^T ||w||^2_{H^2}\, tdt \le C ||w_0||^2_{L^2_{xdx}}
\ee
for the solutions of \eqref{H40}.  
By Proposition \ref{prop20}, 
\be
\label{H61}
\int_0^T ||w||^2_{H^1_0(0,L)} dt \le C ||w_0||^2_{L^2_{xdx}}.
\ee
This yields for $w_0\in L^2(0,L)$
\be
\label{H62}
\int_0^T ||w||^2_{H^1_0(0,L)} dt \le C ||w_0||^2_{L^2}.
\ee
Assume now that $w_0\in {\mathcal D}(A)$, and let $u_0=Aw_0=-w_{0,xxx}-w_{0,x}$.
Denote by $w$ (resp. $u$) the solution of \eqref{H40} issuing from $w_0$ (resp. $u_0$). Then
\[
Aw=-w_{xxx}-w_x=u\in L^2(0,T,H^1_0(0,L)), 
\]
and we infer that $w\in L^2(0,T,H^4(0,L))$. By interpolation, this gives that $w\in L^2(0,T,H^2(0,L))$ if $w_0\in H^1_0(0,L)$, with an estimate of the form
\be
\label{H63}
\int_0^T ||w||^2_{H^2(0,L)} dt \le C ||w_0||^2_{ H^1_0(0,L) }.
\ee
The different constants $C$ in \eqref{H61}-\eqref{H63} may be taken independent of $T$ for $0<T<T_0$. Thus, using Fubini's theorem,  we obtain
\[
\int_0^T s||w(s)||^2_{H^2} ds = \int_0^T\!\!\!  ( \int_t^T ||w(s)||^2_{H^2}ds) dt \le C \int_0^T ||w (t) ||^2_{H^1_0(0,L)} dt \le C ||w_0||^2_{L^2_{xdx}}. 
\]  
This completes the proof of \eqref{H60} and of Theorem \ref{thm11}.
\end{proof}

\subsection{Exact controllability of the nonlinear system}

Our aim is to prove the local exact controllability in $L^2_{\frac{1}{L-x} dx}$ of system  (\ref{ec1}). Note
that the solutions of (\ref{ec1}) can be written as%
\[
u=u_{L}+u_{1}+u_{2}\text{,}%
\]
where $u_{L}$ is the solution of (\ref{ec_lin}) with
initial data $u_{0}  \in L^2_{ \frac{1}{L-x}  dx}$, $u_{1}$ is solution of
\begin{equation}
\left\{
\begin{array}
[c]{lll}%
u_{1,t}+u_{1,x}+u_{1,xxx}= f = (\rho (x)h)_x  &  & \text{in }( 0,T)  \times(  0,L)  \text{,}\\
u_{1}(  t,0)  =u_{1}(  t,L)  =u_{1,x}(  t,L) =0 &  & \text{in }(  0,T)  \text{,}\\
u_{1}(  0,x)  =0 &  & \text{in }(  0,L)
\end{array}
\right.  \label{non_1}%
\end{equation}
with $h=h(  t,x)  \in L^{2}(  0,T;L^2( 0,L)  ) $, 
and $u_{2}$ is solution of
\begin{equation}
\left\{
\begin{array}
[c]{lll}%
u_{2,t}+u_{2,x}+u_{2,xxx}=g(  t,x)  &  & \text{in }(
0,T)  \times(  0,L)  \text{,}\\
u_{2}(  t,0)  =u_{2}(  t,L)  =u_{2,x}(  t,L)
=0 &  & \text{in }(  0,T)  \text{,}\\
u_{2}(  0,x)  =0 &  & \text{in }(  0,L)  \text{,}%
\end{array}
\right.  \label{non_2}%
\end{equation}
with $g=g(t,x)  = - uu_{x}$.

The following result is concerned with the solutions of the non-homogeneous system (\ref{non_2}).

\begin{proposition}
\label{prop_weight} (i) Let $H$ and $V$ be as in \eqref{AAA1}-\eqref{AAA2}
If $u,v\in L^{2}(  0,T;V )$, then $uv_{x}\in L^{1}(  0,T;H)  $. Furthermore, the map%
\[
(u,v)\in L^{2}(  0,T;V)^2  \to uv_{x}\in
L^{1}(  0,T;H)
\]
is continuous and there exists a constant $c>0$ such that
\begin{equation}
\left\Vert uv_{x}\right\Vert _{L^{1}(  0,T;H)  }\leq c\left\Vert u\right\Vert  _{L^{2}( 0,T;V)  }  \left\Vert v\right\Vert  _{L^{2}( 0,T;V)  }  
\text{.} \label{non_3}
\end{equation}
(ii) For $g\in L^{1}(  0,T;H)  $, the mild solution $u$ of \eqref{non_2} given by Duhamel formula satisfies
\[
u_{2}\in C(  \left[  0,T\right]  ;H)  \cap L^{2}(  0,T;V)
=: \mathcal{G}
\]
and we have the estimate 
\be
\label{CCC}
||u_2 ||_{L^\infty (0,T,H) } + ||u _2||_{L^2(0,T,V)} \le C ||g||_{L^1(0,T,H)}.
\ee
\end{proposition}

\begin{proof}
For $u,v\in V$, we have 
\[
||uv_x||_{L^2_{ \frac{1}{L-x}dx }} \le || u ||_{L^\infty} || \frac{v_x}{\sqrt{L-x} }||_{L^2}\le C ||u||_V ||v||_V.  
\]
This gives (i). For (ii), we first assume that $g\in C^1([0,T],H)$, so that $u_2\in C^1([0,T],H)\cap C^0([0,T],{\mathcal D} (A_2))$. Taking the inner product
of $u_{2,t}=A_2u_2+g$ with $u_2$  in $H$ yields
\be
\label{AAA80} 
(u_{2,t},u_2)_H \le -C||u_2||^2_V + C' ||u_2||_H^2 + (g,u_2)_H 
\ee
where $C,C'$ denote some positive constants.
Integrating over $(0,T)$ and using the classical estimate 
\[ 
||u_2||_{L^\infty (0,T,H) } \le C ||g||_{L^1(0,T,H)}  
\]  
coming from semigroup theory, we obtain  (ii) when $g\in C^1([0,T],H)$. The general case ($g\in L^1(0,T,H)$) follows by density.  \end{proof}

Let $\Theta _1 (h):=u_1$ and $\Theta _2(g):=u_2$, where $u_1$ (resp. $u_2$) denotes the solution of \eqref{non_1} (resp. \eqref{non_2}). 
Then  $\Theta_{1}:L^{2}(  0,T;L^2(  0,L)  )  \to \mathcal{G}$ and $\Theta_{2}:L^{1}(  0,T;L^2_{ \frac{1}{L-x}  dx}      )\to \mathcal{G}$  are
well-defined continuous operators, by Propositions \ref{prop45} and \ref{prop_weight}. 

Using Proposition \ref{prop_weight} and the contraction mapping principle, one can prove as in \cite{goubet,mvz,Rosier} 
the existence and uniqueness of
a solution  $u\in {\mathcal G} $ of \eqref{ec1} when the initial data $u_0$ and the forcing term $h$ are small enough. As the proof is similar to those of Theorem 
\ref{exact_nlin}, it will be omitted.  

We are in a position to prove the main result of Section \ref{section4}, namely the (local) exact controllability of system
(\ref{ec1}).

\begin{theorem}
\label{exact_nlin} Let $T>0$. Then there exists $\delta>0$ such that for any
$u_{0}$, $u_{1}\in L_{  \frac{1}{L-x}  dx}^{2}$ satisfying
$\left\Vert u_{0}\right\Vert _{L^2_{  \frac{1}{L-x}  dx}} \leq \delta,\    \left\Vert u_{1}\right\Vert _{L^2_{  \frac{1}{L-x}  dx}}\leq\delta$, 
one can find a control function $h\in L^2(  0,T;L^2 (  0,L)  )  $ such that the solution $u\in {\mathcal G} $ of (\ref{ec1})
satisfies $u(  T,\cdot)  =u_1$ in $(  0,L)  $.
\end{theorem}
As in the linear case, the forcing term $f=(\rho (x)h)_x$ is actually a {\em function} in $L^2_{ (T-t ) dt}(0,T,L^2(0,L))$ supported in $(0,T)\times (L-\nu ,L)$.
\begin{proof}
To prove this result, we apply the contraction mapping principle, following closely \cite{Rosier}. Let $\mathcal{F}$ denote
the nonlinear map%
\[
\mathcal{F}:L^{2}(  0,T; V )  \to \mathcal{G},%
\]
defined by%
\[
\mathcal{F}(  u   )  =u_{L} +\Theta_1 \circ\Gamma(  u_{T}-u_{L}( T,\cdot)  +\Theta_ 2(  uu_{x})  (  T,\cdot) )  - \Theta_2( uu_{x})  \text{,}%
\]
where $u_{L}$ is the  solution of (\ref{ec_lin}) with initial data $u_{0}  \in L^2_{\frac{1}{L-x}dx}$, 
$\Theta_{1}$ and $\Theta_{2}$ are defined as above, and $\Gamma$ is as in Theorem \ref{thm11}.

Remark that if $u$ is a fixed point of $\mathcal{F}$, then $u$ is a solution
of (\ref{ec1}) with the control $h=  \Gamma(  u_{T}-u_{L}( T,\cdot)  +\Theta_ 2(  uu_{x})  (  T,\cdot))$, 
and it satisfies
\[
u(  T,\cdot)  =u_{T},%
\]
as desired.
In order to prove the existence of a fixed point of $\mathcal{F}$, we apply the Banach
fixed-point theorem to the restriction of $\mathcal{F}$ to some closed ball
$\overline{B}(0,R)$ in $L^{2}(  0,T;V) $.

\noindent (i) $\mathcal{F}$ \textit{is contractive. }  Pick any $u,\tilde u\in \overline{B}(0,R)$. 
Using  \eqref{BBB} and     \eqref{non_3}-\eqref{CCC}, 
we deduce that for some constant $C$, independent of $u$,
$\tilde{u}$, and $R$, we have 
\begin{equation}
\left\Vert \mathcal{F}(  u)  -\mathcal{F}(  \tilde{u})
\right\Vert _{L^{2}(  0,T; V)  }%
\leq2CR\left\Vert u-\tilde{u}\right\Vert _{L^{2}(  0,T; V )  }\text{.} \label{fixed1}%
\end{equation}
Hence, $\mathcal{F}$ is contractive if $R$ satisfies
\begin{equation}
R<\frac{1}{4C}\text{,} \label{fixed2}%
\end{equation}
where $C$ is the constant in (\ref{fixed1}).

\noindent\noindent  (ii) $\mathcal{F}$ \textit{maps }$\overline{B}(0,R)$
\textit{into itself.}  Using Proposition \ref{prop40} and the continuity of the operators $\Gamma$, $\Theta_1$, and $\Theta_2$, we infer the existence of a constant $C'>0$ such that 
for any $u\in\overline{B}(0,R)$, we have
\[
\left\Vert \mathcal{F}(  u)  \right\Vert _{L^{2}( 0,T;V)  }\leq C'(  \left\Vert u_{0}%
\right\Vert _{L_{  \frac{1}{L-x}  dx}^{2}}+\left\Vert
u_{T}\right\Vert _{L_{  \frac{1}{L-x}  dx}^{2}}+R^{2})
\text{.}%
\]
Thus, taking $R$ satisfying (\ref{fixed2}) and $R<1/(2C')$ and assuming that $\left\Vert u_{0}\right\Vert
_{L_{ \frac{1}{L-x}  dx}^{2}}$ and $\left\Vert u_{T}\right\Vert
_{L_{  \frac{1}{L-x}  dx}^{2}}$ are small enough, we obtain that the operator
$\mathcal{F}$ maps $\overline{B}(0,R)$ into itself. Therefore the map
$\mathcal{F}$ has a fixed point in $\overline{B}(0,R)$ by the Banach fixed-point Theorem. The proof 
of Theorem \ref{exact_nlin} is complete.
\end{proof}

\noindent\textbf{Acknowledgments:} RC was supported by CNPq and Capes (Brazil)
via a Fellowship, Project `` Ci\^encia sem Fronteiras'' and Agence Nationale
de la Recherche (ANR), Project CISIFS, grant ANR-09-BLAN-0213-02. LR was
partially supported by the Agence Nationale de la Recherche (ANR), Project
CISIFS, grant ANR-09-BLAN-0213-02. AFP was partially supported by CNPq
(Brazil) and the Cooperation Agreement Brazil-France.

\end{document}